\documentclass[12pt, oneside]{scrartcl}
\usepackage[english]{babel}
\usepackage[T1]{fontenc}
\usepackage{makecell}
\usepackage{amssymb}
\usepackage{amsthm}
\usepackage{amsmath}
\usepackage{hyperref}
\usepackage{hhline}
\usepackage{longtable}
\usepackage{amscd}
\usepackage{array}
\usepackage{delarray}
\usepackage{multicol}
\usepackage{makeidx}
\usepackage{enumitem}

\usepackage{float}

\usepackage{tikz}
\usepackage{calc}
\usepackage{adjustbox}
\usetikzlibrary{angles,quotes}
\usetikzlibrary{calc}
\usetikzlibrary{patterns}
\usetikzlibrary{graphs}
\usetikzlibrary{graphs.standard}
\usetikzlibrary{decorations.markings,shapes.geometric,positioning}

\newtheorem{theorem}{Theorem}
\newtheorem*{maintheorem}{Theorem}

\newtheorem{corollary}{Corollary}[theorem]
\newtheorem{conjecture}{Conjecture}
\newtheorem{definition}{Definition}

\newtheorem{lemma}{Lemma}

\newtheorem{observation}{Observation}

\usepackage{etoolbox}
\AfterEndEnvironment{proof}{\par\medskip} 

\usepackage{tikz}

\author{Mahabba El Sahili\footnote{Department of Mathematics, American University of Beirut, Beirut, Lebanon} \and Ayman El Zein\footnote{Computer Science Department, University of Sciences and Arts in Lebanon, Beirut, Lebanon}}

\begin{document} 
\title{Oriented Hamiltonian Paths in Tournaments: Stability under Arc Deletion}
\maketitle

\begin{abstract}
Havet and Thomass\'{e} proved that every tournament of order $n\geq 8$ contains every oriented Hamiltonian path, which was conjectured by Rosenfeld. Recently, it was shown that in any tournament $T$ of order $n\geq 8$, there exists an arc $e$ such that $T-e$ contains any oriented Hamiltonian path. A natural extension of this problem is to study the stability of this property under arbitrary arc deletion. In this paper, we prove that every arc $e$ in a tournament $T$ of order $n\geq 8$ satisfies that $T-e$ contains every oriented Hamiltonian path, except for some explicitly described exceptions.
\end{abstract}


\section{Introduction}
A classical result of R\'{e}dei \cite{Red} asserts that every finite tournament contains an odd number of Hamiltonian directed paths, and thus at least one. Since then, the study of the existence of oriented paths in digraphs, and especially tournaments, has yielded a line of results. In $1971$, Gr\"{u}nbaum \cite{Gr} proved the existence of antidirected Hamiltonian paths in tournaments with the exception of the circuit triangle, the regular tournament of order five, and the Paley tournament of order seven. Subsequently, considering an arbitrary orientation of paths, Rosenfeld \cite{Ros} conjectured the following: There exists an integer $N\geq 8$ such that any tournament of order $n\geq N$ contains any oriented Hamiltonian path. This conjecture was proved in $1986$ by Thomason \cite{Thomason}, who showed the existence of such an $N$ less than $2^{128}$, following a series of partial results \cite{Alsp,For}. In $2000$, Havet and Thomass\'{e} \cite{Hav} settled the problem by proving that $N=8$ is the tight bound. Recently, Bou Hanna \cite{hanna} introduced a new proof of the conjecture for $N=8$.

A natural question that arises is the robustness of such results under modifications such as arc deletion. Thomassen \cite{Thomassen} conjectured that if $I$ is a set of $k-1$ arcs in a $k$-strong tournament $T$, then $T-I$ has a Hamiltonian cycle. This conjecture was proved by Fraisse and Thomassen \cite{Fra}, and recently a stronger result was proved by Bang-Jensen et al. \cite{Ban}. A natural problem that arises is the existence of oriented paths in a tournament once an arc is removed. In $2025$, El Zein \cite{Ez} proved that, in any tournament $T$ of order $n\geq 8$, there exists an arc $e$ such that $T-e$ contains every oriented Hamiltonian path. In particular, this arc joins vertices having a maximum out- and in-degree. Independently, in a seminar on Graph Theory, El Sahili \cite{ES} conjectured the following.
\begin{conjecture}
   If $T$ is a strong tournament of order $n\geq 8$, $e$ is an arc of $T$, and $P$ is an oriented path of order $n$, then $T-e$ contains $P$. 
\end{conjecture}

Indeed, if $T$ is not strong, this may not hold. For instance, let $e=(x,y)$ and $P=v_1\cdots v_n$. Take the following two cases:
 \begin{enumerate}[label=(\roman*)]
     \item $P$ is directed, $x$ and $y$ have the same inneighborhood and outneighborhood in $T-e$, and every vertex of $N^-(x)-\{y\}$ dominates every vertex of $N^+(x)-\{y\}$
     \item $P$ has exactly two blocks,
     $x$ and $y$ are sinks in $T-e$ when $(v_1,v_2)$ is a forward arc, and $x$ and $y$ are sources in $T-e$ otherwise.
 \end{enumerate}
 
Clearly, in these cases, which we call \textit{special exceptions}, $T-e$ does not contain $P$.

In this paper, we provide a complete characterization of when the deletion of an arbitrary arc preserves the existence of any oriented Hamiltonian path. We show that the only obstructions are the two explicit special exceptions, as stated in the following theorem.

\begin{maintheorem}
Let $T$ be a tournament of order $n\geq 8$, $P$ be a path of order $n$, $x,y\in V(T)$, and $D$ the digraph obtained from $T$ by deleting the arc joining $x$ and $y$. Then, $D$ contains $P$ if and only if $(T,P,\{x,y\})$ is not a special exception.
\end{maintheorem}

We work with simple graphs, that is, graphs without loops or multiple edges. A digraph is a graph in which every edge has an orientation; oriented edges are called arcs. A tournament is an orientation of a complete graph. For a digraph $D$, we denote its vertex set by $V(D)$ and its arc set by $E(D)$, and the order of $D$ is $|V(D)|$. If $(x,y) \in E(D)$, we say that $x$ dominates $y$ and that $y$ is dominated by $x$. We also say that $x$ is an inneighbor of $y$ and $y$ an outneighbor of $x$. The inneighborhood of $x$ in $D$, denoted by $N^-_D(x)$, is the set of inneighbors of $x$ in $D$, and the outneighborhood of $x$ in $D$, denoted by $N^+_D(x)$, is defined likewise. We denote by $d^-_D(x)=|N^-_D(x)|$ and $d^+_D(x)=|N^+_D(x)|$. We say that a vertex $u\in D$ is a source (resp., sink) if $d^-_D(u)=0$ (resp., $d^+_D(u)=0$). For $X,Y\subseteq  V(D)$, we say that $X$ dominates $Y$ if every vertex of $X$ dominates every vertex of $Y$. The dual $\overline{D}$ of $D$ is obtained by reversing all arcs of $D$. If $D$ is a digraph and $X \subseteq V(D)$, we write $D[X]$ for the subgraph induced by $X$, and $D-X$ for $D[V(D)\setminus X]$. For a single vertex $u \in V(D)$, we write $D_u = D-u$.

An oriented path, which we will simply call a path, is a sequence of distinct vertices $P = v_1 \cdots v_n$ such that each consecutive pair forms an arc. The vertex $v_1$ is the origin and $v_n$ the end of $P$. The length of $P$ is $n-1$, and its order is $n$. We denote by $V(P)=\{v_1,\dots,v_n\}$. The reversed path is denoted $P^{-1} = v_n \cdots v_1$, and the path obtained by removing the origin is denoted $^{*}P = v_2 \cdots v_n$. For clarity, we write $^*P^{-1}$ to mean $^*(P^{-1})=v_{n-1}\cdots v_1$. A path is an outpath (resp., inpath) if its origin dominates (resp., is dominated by) its second vertex. It is directed if all arcs are consistently oriented. A block of a path $P$ is a maximal directed subpath of $P$. In general, an oriented path is formed by successive blocks. The $i^{th}$ block of $P$ is denoted by $B_i(P)$ while its length is denoted by $b_i(P)$. The type of an outpath $P$ is the sequence $+(b_1(P),b_2(P),\dots,b_s(P))$ while that of an inpath $Q$ is $-(b_1(Q),b_2(Q),\dots,b_t(Q))$. Occasionally, we write $P=\pm (b_1(P),\dots,b_s(P))$ instead of $P$ is of type $\pm (b_1(P),\dots,b_s(P))$. When two paths $P$ and $Q$ are of same type, we write $P\equiv Q$. Given a path $P$ and a digraph $D$, we denote by $O_D(P)$ the set of origins of paths $Q \equiv P$ in $D$. A path is Hamiltonian if it passes through all vertices of the digraph.

A digraph $D$ is strong if for every pair of vertices $x,y \in V(D)$ there exists a directed outpath with origin $x$ and end $y$. A strong component of a digraph is a strong subdigraph that is maximal with respect to the inclusion. For $X \subseteq V(D)$, the outsection $S^+(X)$ (resp., insection $S^-(X)$) is the set of vertices $y$, that is the end of a a directed outpath (resp., inpath) with origin in $X$.  We simply write $S^+(x_1,\dots,x_t)$ instead of $S^+(\{x_1,\dots,x_t\})$. We also write $s^+(X) = |S^+(X)|$ and $s^-(X) = |S^-(X)|$. A vertex $x$ is an outgenerator (resp., ingenerator) of $D$ if $S^+(x) = V(D)$ (resp., $S^-(x) = V(D)$).

\section{Proof of the Main Theorem}
 
Prior to establishing Rosenfeld’s conjecture, Thomason \cite{Thomason} proved that every tournament $T$ on \(n+1\) vertices contains any oriented path $P$ of order \(n\). In particular, any subset of \(b_1(P)+1\) vertices of $T$ must include an origin of $P$. Havet and Thomass\'{e} \cite{Hav} improved this result by demonstrating a stronger property: for any pair of vertices \(x,y\in T\), at least one is an origin of \(P\) whenever 
$s^{+}(x,y) \geq b_1(P)+1$
in the case where \(P\) is an outpath, and 
$s^{-}(x,y) \geq b_1(P)+1$
otherwise. They aimed to establish the same result for Hamiltonian paths.
However, they identified several tournaments that fail to meet the hypothesis. Consequently, they defined a pair \((T;P)\) to be an \emph{exception} if the tournament \(T\) admits two vertices \(x, y\) with 
$s^{+}(x,y) \geq b_1(P)+1$ 
while neither vertex serves as an origin of the outpath \(P\). By duality, one may equivalently consider the insection induced by \(x\) and \(y\) in place of the outsection in case of an inpath $P$. Despite the presence of \(68\) explicitly characterized exceptions, described in Appendix \ref{appendix:A}, Havet and Thomass\'{e} proved the following theorem.

\begin{theorem}\label{HT} \cite{Hav}
Let $T$ be a tournament of order $n$ and $P$ be an outpath of order $n$. If $x$ and $y$ are two vertices of $T$ such that $s^+(x,y)\geq b_1(P)+1$, then one of the following holds:\begin{itemize}
\item[(i)] $x$ or $y$ is an origin of $P$ in $T$.
\item[(ii)] $(T;P)$ is an exception.
\end{itemize}
\end{theorem}
Based on this, they proved that for a tournament $T$ of order $n$ and a path $P$ of order $n$, unless $(T;P)$ is one of Gr\"{u}nbaum's exceptions, $T$ contains $P$. This implies the tight bound of $N=8$ for Rosenfeld's conjecture.
The following corollary to Theorem \ref{HT} is immediate, and will be used throughout our proof.

\begin{corollary}\label{corollary1}
Let $T$ be a tournament of order $n$ and $P$ be a path of order $n$ where $(T;P)$ is not an exception. If $T$ is strong or $b_1(P)=1$, then for any two vertices of $T$, one of them is an origin of $P$.

\end{corollary}
El Zein \cite{Ez} proved the following corollary, which we will also employ.
\begin{corollary}\label{corollary2}
 Let $T$ be a tournament of order $n$ and $P$ be an outpath (resp. inpath) of order $n$. If ($T; P)$ is not an exception, then $|O_T(P)|\geq n-b_1(P)$. In particular, if $P$ is not directed, then $|O_T(P)|\geq 2$.
\end{corollary}
Furthermore, by counting the path origins appearing in the exceptions listed in Appendix \ref{appendix:A}, we establish a lower bound on the number of origins of $P$ that occur in the majority of these exceptions.
\begin{observation}\label{obs1}
Suppose that $(T;P)$ is an exception different from those of Gr\"{u}nbaum and $Exc1,\allowbreak Exc7, \allowbreak E_{1}(4)$ or their duals. Then, the following hold.
\begin{enumerate}[label=(\roman*)]
     \item $|O_T(P)|\geq 2$.
     \item If $|T|\geq 6$ and $(T;P)$ is not one of $Exc15$, $Exc24$ or their duals, then $|O_T(P)|\geq 3$.
     \item If $|T|\geq 7$, then $|O_T(P)|\geq 4 $.
 \end{enumerate}
\end{observation}
We also notice that for most exceptions $(T;P)$, the tournament $T$ is strong. For clarity, we write $E_i$ instead of $E_i(n)$ which is the notation used in Appendix \ref{appendix:A}.
\begin{observation}\label{obs2}
If $(T;P)$ is an exception where $T$ is not strong, then $(T;P)$ is one of the exceptions $E_i,E_j'$, $i\in \{1,2,3,4,5,6,8,9,10\}$, $j\in \{8,9,10\}$ or their duals.
\end{observation}
Let $T$ be a tournament and $x,y\in V(T)$. We first consider the case where $T_x=T-x$ is an exception and $d^+_D(x)\geq 2$. This is done by noticing key properties of the exceptions in Appendix \ref{appendix:A}, that reduce the number of verifications needed. 
\begin{lemma}\label{lemma1}
 Let $T$ be a tournament of order $n\geq 8$, $P$ be an outpath of order $n$, $x,y\in V(T)$, and 
 $D$ the digraph obtained from $T$ by deleting the arc joining $x$ and $y$. If $d^+_D(x)\geq 2$ and $(T_x;^*P)$ is an exception, then $D$ contains $P$.
\end{lemma}
\begin{proof}
Let $P=v_1\cdots v_n$. As in Appendix \ref{appendix:A}, let $S$ be the set of
non-origins of $^*P$ in $T_x$. If there exists some $v$ in $N^+_D(x)\setminus S$, then $v$ is the origin of a path $Q\equiv ^*P$ in $T_x$, and $xQ\equiv P$. Thus, we can assume that $N^+_D(x)\subset S$. Consider the following observation.
\begin{observation}\label{obs3}
Suppose $(T_x;^*P)$ is an exception with $(T_x,^*P)\notin \{Exc33, DualExc33, \allowbreak  Dual Exc37, Exc39, Dual Exc49, E'_8,Dual E_8',E_{12}(7),E_i,DualE_i\}$ for\\$i\in\{1,2,3,4,5,6\}$. Then, $(T_x;^*P)$ satisfies the following conditions:  \begin{enumerate}[label=(\roman*)]
     \item $(T_x;^*P^{-1})$ is not an exception where $^*P^{-1}=v_{n-1}\cdots v_1$.
     \item $|T_x\setminus S|\geq 2$
     \item Either $T_x$ is strong or $b_1(^*P^{-1})=1$. 
 \end{enumerate}
\end{observation}
\noindent Therefore, $d^-_D(x)\geq 2$. If $P^{-1}$ is an outpath (resp., inpath), take $z$ and $w$ to be two vertices of $N^+_D(x)$ (resp., $N^-_D(x)$). Then, by Corollary \ref{corollary1}, $z$ or $w$ is an origin of $^*P^{-1}$ in $T_x$. Therefore, we have $xQ\equiv P^{-1}$. This shows that $D$ contains $P^{-1}$ so it contains $P$. Now we consider the exceptions that do not satisfy Observation \ref{obs3}.\\
If $(T_x;^*P^{-1})$ is $Exc33$, we also have that $(T_x;^*P^{-1})$ is not an exception, and $T_x$ is strong. Moreover, $P^{-1}$ is an outpath and $d^+_D(x)\geq 2$. Therefore, as in the above $D$ contains $P^{-1}$ so it contains $P$.\\
If $(T_x,^*P^{-1})$ is the dual of $Exc33$, let $u,v \in N^+_D(x)$.
Assume wlog that $u=1$. If $v=3$, then $61x27435\equiv P$. The other choices for $v$ can be handled similarly.\\
If $(T_x;^*P^{-1})$ is $Exc39$, suppose that $y=7$. Then, $(T_y;^*P)$ is not an exception, and $d^+_D(y)\geq 3\geq b_1(^*P)+1$. Thus, by Theorem \ref{HT}, an outneighbor of $y$ is the origin of a path $Q\equiv ^*P$ in $T_y$. Therefore, $yQ\equiv P$. Now if $y\neq 7$, let $u=7$ and $A=T_x-u$. We have $v,w\in N^+_D(x)$ with $s^+_{A}(v,w)\geq 3=b_{1}(^*P)+1$. Thus, either $v$ or $w$ is an origin of $Q\equiv ^*P$ in $A$. Then, $uxQ\equiv P$.\\
If $(T_x;^*P)$   is one of the exceptions $DualExc37$ or $DualExc49$, we have $d^-_D(x)\geq 2$. Moreover, there exists some $u\in T-\{x,y\}$ where $T-\{x,u\}$ is strong, and $(T-\{x,u\}, ^{**}P)$ is not an exception. Simply take $u\in \{1,2\}-y$ in the case of the dual of $Exc 37$ and take $u\in \{4,5\}-y$ otherwise.
Therefore, by Corollary \ref{corollary1} $T-\{x,u\}$ contains a path $Q\equiv ^{**}P$ of origin $v$ in $N^-_D(x)$, and $uxQ\equiv P$.\\
If $(T_x;^*P)$ is exception $E_8'(n)$, $P^{-1}=-(1,2,|T_{x}|-3)$. Take $u,v\in 3A-y$. Then, $s^+(u,v)\geq 3$, and thus $u$ or $v$ is an origin of $Q\equiv ^{*}P$ in $T_x$. Therefore, $xQ\equiv P^{-1}$.\\
If $(T_x;^*P)$ is the dual of exception $E_8'(n)$, also take $u,v\in 3A-y$. We have that $s^-_{T_x}(u,v)\geq |T_x|$ and thus one of them is the origin of a path $Q\equiv P^{-1}$ in $T_x$. We have $xQ\equiv P$.\\
If $(T_x;^*P)$ is exception $E_{12}(7)$, $P=+(3,1,3)$. Let $u\in \{2,3\}-y$, then we have $u\in N^-_D(x)$. Moreover, $(T_x-u;^{**}P)$ is not an exception and $b_{1}(^{**}P)=1$. Therefore, by Corollary \ref{corollary1}, $T_x-u$ contains a path $Q\equiv ^{**}P$ with origin in $N^+_D(x)$. Thus, $uxQ\equiv P$.\\
Suppose that $(T_x;^*P)$ is one of exceptions $E_i$, $i\in \{1,2,3,4,5,6\}$.
we have $2\leq b_1(P)\leq 3$ and $P$ has two blocks. If $b_1(P)=2$, let $v\in T_x-S\cup\{y\}$, and observe that an outneighbor of $x$ is an origin of an inpath $Q$ in $T_x-v$, and we have $vxQ\equiv P$. Likewise, if $b_1(P)=3$, let $v_1,v_2\in T_x-S\cup\{y\}$ with $v_2\in N^+(v_1)$. Let $Q$ be an inpath in $T_x-\{v_1,v_2\}$ of origin in $ N^+_D(x)$, and we have $v_1v_2xQ\equiv P$.\\
Finally, suppose that $(T_x;^*P)$ is one of the duals of exceptions $E_i$, $i\in \{1,2,3,4,5,6\}$. We have that $P^{-1}$ is an inpath. Moreover, $d^+_D(x)\leq 2$ for all these exceptions other than $E_1$. Then, note that there exists exists $z,w\in N^-_D(x)$ with $s^-_{T_x}(z,w)\geq b_1(P^{-1})+1$ for all exceptions $E_i$, $i\in \{1,2,3,4,5,6\}$. Thus, by Theorem \ref{HT}, $z$ or $w$ is the origin of a path $Q\equiv ^{*}P^{-1}$ in $T_x$. Then, $xQ\equiv P^{-1}$ as desired.
\end{proof}
Note that Lemma \ref{lemma1} also applies to an inpath $P$ but with $d^-_D(x)\geq 2$. We now precisely define a \textit{special exception} in the case of an outpath $P$. The corresponding definition for an inpath is obtained analogously by taking the duals of the digraphs.
\begin{definition}
Let $T$ be a tournament of order $n$, $P$ be an outpath of order $n$, $x,y\in V(T)$, and 
 $D$ the digraph obtained from $T$ by deleting the arc joining $x$ and $y$.  We call $(T,P,\{x,y\})$ a \emph{special exception} if one of the following holds:
 \begin{enumerate}[label=(\roman*)]
     \item $P$ is directed, $N^-_{D}(x)=N^-_{D}(y)$, $N^+_{D}(x)=N^+_{D}(y)$ and $N^-_{D}(x)$ dominates $N^+_{D}(x)$. Note that $N^-_{D}(x)$ and $N^+_{D}(x)$ may be empty.
     \item $P$ has exactly two blocks, and
     $x$ and $y$ are sinks in $D$.
 \end{enumerate}
\end{definition}
We will now prove the main theorem, which we restate below.
\begin{theorem}\label{thm:main}
 Let $T$ be a tournament of order $n\geq 8$, $P$ be a path of order $n$, $x,y\in V(T)$, and 
 $D$ the digraph obtained from $T$ by deleting the arc joining $x$ and $y$. Then, $D$ contains $P$ if and only if $(T,P,\{x,y\})$ is not a special exception.
\end{theorem}
\begin{proof}Without loss of generality, suppose that $P$ is an outpath. It is clear that if $(T,P,\{x,y\})$ is a special exception, then $D$ does not contain $P$. For the sufficient condition, suppose that $(T,P,\{x,y\})$ is not a special exception. 
Assume without loss of generality that $d^+_D(x)\geq d^+_D(y)$. Let $P=v_1\cdots v_n$. First, suppose that $d^+_D(x)=0$. Then, $d^+_{D}(y)=0$. Now, since $(T,P,\{x,y\})$ is not a special exception, $P$ has at least three blocks. Let $v_{i}$ and $v_j$ be the ends of the first and third blocks of $P$, respectively. Define a path $P'$ obtained by deleting $v_i$ and $v_j$ from $P$ and adding an arc between $v_{i-1}$ and $v_{i+1}$, together with an arc between $v_{j-1}$ and $ v_{j+1}$ if $j< n$. We can do so such that $P'$ is not antidirected. Then, $T-\{x,y\} $ contains a path $ u_{1}\cdots u_{i-1}u_{i+1}\cdots u_{j-1}u_{j+1}\cdots u_{n}\equiv P'$. Hence, $u_{1}\cdots u_{i-1}xu_{i+1}\cdots u_{j-1}yu_{j+1}\cdots u_{n}\equiv P$. Thus, we can assume that $d^+_D(x)\geq 1$. We will study two cases.\\\\
\textbf{Case 1:} $b_{1}(P)\geq 2$.
    \\ First, suppose that $P$ is directed.  If $y$ is dominated by a vertex in $N^+_{D}(x)$, then there exists a Hamiltonian directed outpath $Q$ of the tournament $D[N^+_D(x)\cup \{y\}]$ with origin different from $y$. Let $Q'$ be a Hamiltonian directed outpath of $D[N^-_D(x)]$.
    Hence, $Q'xQ\equiv P$. Therefore, we can assume that $y$ dominates every vertex of $N^+_{D}(x)$. Similarly, if $N^-_{D}(x)$ is non-empty, every vertex of $N^-_{D}(x)$ dominates $y$. Since $(T,P,\{x,y\})$ is not a special exception, there exist $u\in N^-_{D}(x)$ and $v\in N^+_{D}(x)$ such that $v$ dominates $u$. Let $L$ (resp., $L'$) be a directed outpath formed of the vertices of $N^-_{D}(x)-u$ (resp., $N^+_{D}(x)-v$), these paths are possibly empty. We have $LxvuyL' \equiv P$. Thus, in the following, we can assume that $P$ is not directed, and we consider two subcases.\\
    \\ \textbf{Subcase 1.1:} $d^+_{D}(x)\geq 2$.\\
    If $(T_x;^*P)$ is an exception, then, by Lemma \ref{lemma1}, $D$ contains $P$. So, we may assume that $(T_x;^*P)$ is not an exception.
    If $x$ dominates an origin of a path $Q\equiv ^*P$ in $T_x$ different from $y$, then $xQ\equiv P$. Otherwise, by Theorem \ref{HT}, we can assume that $d^+_D(x)\leq  b_1(^*P)<b_{1}(P)$. Moreover, $T_x$ is not strong.\\
    Let $m=b_{1}(P)-d^+_{D}(x)$. Clearly, $m>0$. Since $P$ is not directed, $d^-_D(x)\geq m$. Let $M$ be an arbitrary subtournament of $D[N^-_D(x)]$ on $m$ vertices and let $Q$ be a Hamilatonian directed outpath of $M$. Let $P'=v_{m+2}\cdots v_{n}$ and $N=T_x-M$. Note that we have $m+2\leq b_1(P)$. Hence, $P'$ is an outpath. Moreover,
    \begin{equation*}
        b_1(P')=b_1(P)-(m+1)=d^+_D(x)-1.
    \end{equation*}
    Hence, $d^+_D(x)=b_1(P')+1$. If $(N;P')$ is not an exception, then there exists some $u\in N^+_D(x)$ that is an origin of a path $Q'\equiv P'$ in $N$. Thus, $QxQ'\equiv P$.
    Now, we can assume that $(N;P')$ is an exception and $N^+_D(x)\subset S$, where $S$ is the set of vertices in $N$ that are not origins of a path $P'$. If $V(N)\setminus \{y\}\subset N^+_D(x)$, then $b_1(P')=d^+_D(x)-1= |N|-2$ and $|S|\geq |N|-1$. This is only possible if $N=3A$ and $P=+(n-2,1)$. As in the case where $P$ is directed, we can assume that $y$ is dominated by all vertices of $Q$. Let $u_1$ (resp., $u_2$) be the outneighbor (resp., inneighbor) of $y$ in $N$. We have $Qyu_1u_2x\equiv P$.\\
    Now, we can assume that $N^-_{D}(x)\cap V(N)\neq \emptyset$. As $M$ is arbitrarily chosen and $T_x$ is not strong, we can suppose that $N$ is not strong. Thus, $(N;P')$ is one of the exceptions of Observation \ref{obs2}.\\
    If $(N;P')$ is one of the exceptions $E_1,E_3,E_5$, then $P=+(m+2,n-m-3)$. Let $a \in N\cap N^-_D(x)$. It is seen that $x$ dominates an ingenerator $b$ of $N-a$. Let $L$ be a Hamiltonian directed inpath of $N-a$ of origin $b$ and $Q'$ be a Hamiltonian directed outpath of $D[V(M)\cup \{a\}]$. Then, $Q'xL\equiv P$.\\
    If $(N;P')$ is one of the exceptions $E_2,E_4,E_6$, then $d^+_D(x)=b_1(P')+1=3$ and $|S|=2$, which contradicts the fact that $N^+_D(x)\subseteq S$.\\
    If $(N;P')$ is one of the exceptions $E_8,E'_8,E_9,E'_9,E_{10},E'_{10}$, then we can assume that $(T_x;^*P^{-1})$ is not an exception. Indeed, otherwise we can apply Lemma \ref{lemma1} to show that $D$ contains $P^{-1}$ and so it contains $P$.
    By Corollary \ref{corollary1}, if $P^{-1}$ is an outpath (resp., inpath), then $x$ dominates (resp., is dominated by) an origin of a path $R\equiv ^*P^{-1}$ in $T_x$. Thus, $xR\equiv P^{-1}$.\\
    \\ \textbf{Subcase 1.2:} $d^+_{D}(x)=1$.\\ 
   Let $z$ be the unique outneighbor of $x$. Let $A=V(D)\setminus \{x,y,z\}$.\\
   We will first prove the case where $P$ has exactly two blocks. Suppose that $P$ has exactly two blocks, and, without loss of generality, assume that $b_1(P)\geq 3$. If $y$ dominates $z$, then every vertex in $A$ dominates $y$. Let $L$ be a directed outpath of order $b_1(P)-1$ in $D[A]$ and $L'$ be a Hamltonian directed inpath of $D[A\setminus V(L)]$. We have $LxzyL'\equiv P$. So, we may assume that $z$ dominates $y$. Let $L$ be a directed outpath of order $b_1(P)-2$ in $D[A]$ that contains the outneighbor of $y$ if it exists. Consider a Hamiltonian directed inpath $L'$ of $D[A\setminus V(L)]$. We have $LxzyL'\equiv P$.\\
   Now, we will assume that $P$ has at least three blocks. First, suppose that $N^+_D(z)\cap A\neq \emptyset$. Let $m=\min(|N^+_D(z)\cap A|,b_{1}(P)-1)$. Clearly, $m>0$. Consider an arbitrary subtournament $M$ of $D[N^+_D(z)\cap A]$ of order $m$. Let $Q$ be a directed Hamiltonian outpath of $M$. Let $M'$ be a subtournament of $D[N^-_D(z)\cap A]$ of order $b_1(P)-m-1$, which may be empty, and let $R$ be a Hamiltonian directed outpath of $M'$. Note that $RzQx$ is a directed outpath of length $b_1(P)$. Let $N$ be the subtournament of $D$ induced by $(A\cup \{y\})\setminus (V(Q)\cup V(R))$. Let $P'=v_{b_1(P)+2}\cdots v_n$.\\
   Suppose that $P'$ is directed. Then, it is an outpath. As $d^+_D(y)\leq 1$ and $|A|\geq 5$, $N$ has a Hamiltonian directed outpath $L$ of origin distinct from $y$. Thus, $RzQxL\equiv P$.\\
   Now, suppose that $P'$ is not directed. If $(N;P')$ is not an exception, then, by Corollary \ref{corollary2}, $|O_{N}(P')|\geq 2$. Moreover, if $(N;P')$ is an exception different from those of Observation \ref{obs1}, then $|O_N(P')|\geq 2$. Thus, in both of these cases, $N$ contains a path $L\equiv P'$ with origin different from $y$. Therefore, $RzQxL\equiv P$. So, we may assume that $(N;P')$ is one of the exceptions of Obseration \ref{obs1}. Moreover, if $y$ admits a unique outneighbor $w$ in $A$ (different from $z$), then we can assume without loss of generality that $w\in N^+_D(z)$. And since $M$ is arbitrarily chosen, we can further assume that $w\in M$. Thus, we can always suppose that $y$ is a sink in $N$. This leaves us with $(N;P')$ being the dual of exception $E_1(4)$.
   Observe that $P^{-1}$ is an inpath with $b_1(P^{-1})=2$. Thus, since $d^-_D(x)\geq 2$ and $|T_x|\geq 7$, by Corollary \ref{corollary1} and Observation \ref{obs1}, $T_x$ contains a path $L\equiv ^*P^{-1}$ with an origin in $N^-_D(x)$. Then, $xL\equiv P^{-1}$.\\
   Finally, we have the case where $N^+_D(z)\cap A= \emptyset$. Without loss of generality, we may suppose that $N^+_D(y)\cap A=\emptyset$. Indeed, otherwise, an outneighbor of $y$ in $A$ may play the role of $z$ in order to return to the case where $N^+_D(z)\cap A\neq \emptyset$. 
   Let $Q$ be a directed outpath in $D[A]$ of order $b_1(P)-1$ and let $P'=v_{b_1(P)+2}\cdots v_n$. Let $M=D[(A\setminus V(Q))\cup\{y\}]$. Since $Q$ is arbitrarily chosen, we can suppose that $M$ is different from $ F_1(4)$ and its dual. Note that $P'$ is not a directed inpath, since otherwise $P$ would have only two blocks. Moreover, $y$ is a sink in $M$. Thus, by Corollary \ref{corollary2} and Observation \ref{obs1},
    $M$ contains a path $L\equiv P'$ with origin different from $y$. Therefore, $QxzL\equiv P$. Case 1 is done.\\\\
\textbf{Case 2:} $b_{1}(P)=1$.\\
We also consider two subcases.\\\\
\textbf{Subcase 2.1:} $d^+_{D}(x)\geq 2$.\\
If $(T_x;^*P)$ is an exception, then, by Lemma \ref{lemma1}, $D$ contains $P$. So, we may assume that $(T_x;^{*}P)$ is not an exception. If $x$ dominates an origin of a path $Q\equiv ^*P$ in $T_x$ different than $y$, then $xQ\equiv P$. Otherwise, by Theorem \ref{HT}, we can assume that $s^-_{T_x}(N^+_D(x))\leq b_2(P)$. Hence, $2\leq d^+_D(x)\leq b_2(P)$.\\
Suppose that $P=+(1,n-3,1)$. Let $X=S^-_{T_x}(N^+_D(x))$, by the above we have $|X|\leq n-3$. Let $M=D[V(D)\setminus (X\cup \{x,y\})]$. If $y$ has at least two outneighbors in $M$, then $s^-_{T_y}(N^+_D(y))\geq n-1$, and the result follows. So, we may assume that $y$ has at most one neighbor in $M$. If $y\in X$, then $|V(M)|\geq 2$, and $M\subset N^+(y)$ which is a contradiction. Therefore, we have $y\notin X$. Now, it is clear that $X$ contains an origin of a path $R=+(n-3,1)$ in $T_y$. Then, $yR\equiv P^{-1}$.\\
Now, assume that $P\neq +(1,n-3,1)$. Let's suppose that $b_2(P)\leq \frac{n-3}{2}$ and at the end of the subcase we will say why it is sufficient to suppose that. Then,\begin{equation*}\begin{aligned}
        d^-_D(x) &\geq (n-2) - b_2(P) \\ & \geq (n-2) - \frac{n-3}{2}\\ & \geq \frac{n-1}{2} \\ & \geq b_2(P) + 1.\end{aligned}
\end{equation*}\label{eq 1}
In particular, $d^-_D(x)\geq 3$. Let $u,w,z\in N^-_D(x)$ such that $w$ is an ingenerator of $D[N^-_D(x)\setminus \{u\}]$ and let $A=T_x-u$. Then,
\begin{equation*}
    \begin{aligned}
        s^-_{A}(z,w) &\geq |N^-_D(x)\setminus \{u\}| \\ &\geq b_2(P) \\ & = b_1(^{**}P) + 1.
    \end{aligned}
\end{equation*}
If $z$ or $w$ is an origin of a path $Q\equiv ^{**}P$ in $A$, then $uxQ\equiv P$. Otherwise, by Theorem \ref{HT}, $(A;^{**}P)$ is an exception. Moreover, we can assume that $N^-_D(x)\setminus u\subset S$, where $S$ is the set of vertices in $A$ that are not origins of $^{**}P$.
Therefore, $|S|\geq |N^-_D(x)\setminus \{u\}|\geq \frac{n-1}{2}-1\geq \frac{n-3}{2}$. In particular, $|S|\geq 3$. Also note that if $(A;^{**}P)$ is a finite exception, then $A$ is strong, and so, \begin{equation*}
 \begin{aligned}
    s^-_{T_x}(N^+_D(x))&\geq |A|\\ &=n-2
    \\&\geq b_{2}(P)+1    \end{aligned}
\end{equation*}
Which contradicts the above. We are left with $A$ being one of the duals of $F_1,F_8,F_9,F_{10}$. If $A$ is one of the duals of $F_8,F_9,F_{10}$, then $x$ dominates the ingenerators of $A$. Hence, $s^-_{T_x}(N^+_D(x))>b_2(P)$, a contradiction. So, we may assume that $(A;^{**}P)$ is the dual of $E_1(n)$. Let $P'=\overline{P^{-1}}=+(n-4,2,1)$. By Case 1, $\overline{D}$ contains $P'$. Then, $D$ contains $P$.\\
 In the above, we mentioned that it is sufficient to assume that $b_2(P)\leq \frac{n-3}{2}$. Indeed, suppose that $b_2(P)>\frac{n-3}{2}$. First, suppose that $P^{-1}$ is an outpath. If $b_1(P^{-1})\geq 2$, then the result follows by Case 1. Otherwise, $b_1(P^{-1})=1$. Then, $b_2(P^{-1})\leq \frac{n-3}{2}$, and we are done. Finally, assume that $P^{-1}$ is an inpath. Let $P'=\overline{P^{-1}}$. Again, if $b_1(P')\geq 2$, then $\overline{D}$ contains $P'$. Hence, $D$ contains $P$. Otherwise, $b_1(P')=1$. As $P\neq +(1,n-3,1)$, we have $b_2(P')\leq \frac{n-3}{2}$. Moreover, $d^+_{\overline{D}}(x)=d^-_D(x)\geq 2$. Thus, by the above $\overline{D}$ contains $P'$. Therefore, $D$ contains $P$.\\
\\ 
\textbf{Subcase 2.2:} $d^+_{D}(x)=1$.\\
Let $z$ be the unique outneighbor of $x$. Let $B=V(D)\setminus \{x,y,z\}$. First, suppose that $P=+(1,n-3,1)$. Assume that $y$ dominates $z$. Since  $|B|\geq 5$, $B\neq 3A$. Then, $B$ contains a path $Q\equiv v_4\cdots v_n$. We have $yzxQ\equiv P$. Now suppose that $z$ dominates $y$. Let $w$ be an inneighbor of $y$ different than $x$ and $z$. As in the above, $B-w$ contains a path $Q\equiv v_5\cdots v_n$ and $wyzxQ\equiv P$.\\ 
Now we can assume that $P\neq +(1,n-3,1)$. Let $w\in B$ and let $M=D[T-\{x,w\}]$. For clarity, denote by $P'=^{**}P$. Note that $(M; P')$ cannot be $Exc33, Dual Exc33,Exc15$ or $DualExc24$. That is because $y\in M$, and $d^+_D(y)\leq 1$. Suppose that $(M, P')$ is not an exception, and $b_{1}(P')\leq |M|-3$. Then, by Corollary \ref{corollary2}, $O_{M}(P')\geq |M|-b_{1}(P')\geq 3$. Moreover, if $(M, P')$ is an exception that is different from $Exc24$ and the dual of $Exc15$, we also have  $O_{M}(P')\geq3$ by Observation \ref{obs1}. Therefore, in both cases, there exists some $u$ in $M$, different from $y$ and $z$, with $u$ origin of $Q \equiv P'$ in $M$. Thus, we have $wxQ\equiv P$.\\
If $(M,P')$ is $Exc24$, then $y$ must be vertex $4$ which is in $S$, so there exists some $ u $ in $M-\{y,z\}$ with $u$ origin of $Q\equiv P'$ in $M$. Thus, as in the above $wxQ\equiv P$.\\
If $(M,P')$ is $DualExc15$, let $Q=^*P^{-1}$. Then, by Observation \ref{obs1}, we have $O_{T_x}(Q)\geq 3$. Therefore, $ T_{x}$ contains $ Q\equiv ^{*}P^{-1}$ with origin different from $z$ and $y$. Hence, since $P^{-1}$ is an inpath, we have $xQ\equiv P^{-1}$.\\
Finally, suppose that $(M, P')$ is not an exception and $b_{1}(P')\geq |M|-2$. Therefore, $P$ is one of $+(1,1,m-2,1)$, $+(1,1,m-1)$, and $+(1,m-1)$. Notice that we either have $b_1(P^{-1})\geq 2$ or $b_{1}(^{**}P^{-1})\leq |M|-3$. These cases were treated above, and thus $D$ contains $P^{-1}$ so it contains $P$.
\end{proof}
\begin{corollary}
Let $T$ be a tournament of order at least $8$ and $e$ be an arc in $T$. If $T$ is strong, then $T-e$ contains every oriented Hamiltonian path.
\end{corollary}
\begin{corollary}
Let $T$ be a tournament of order at least $8$. If $\min (\delta^-(T),\delta^+(T))\geq 1$, then, for every arc $e$, $T-e$ contains every non-directed Hamiltonian path.
\end{corollary}
\begin{corollary}
Let $T$ be a tournament of order $n\geq 8$ and $P$ be a path of order $n$ having at least three blocks. For every arc $e$, $T-e$ contains $P$.
\end{corollary}


\appendix
\section{The exceptions defined by Havet and Thomass\'{e}}\label{appendix:A}
Havet and Thomass\'{e} \cite{Hav} established two categories of exceptions: the finite exceptions and the infinite families of exceptions. The notation of an exception is the following: $[T;P;S;P_1,...,P_k]$ where $T$ is a tournament illustrated in Figures \ref{figure 1}, \ref{figure 2} and \ref{figure 3}, $P$ is an outpath, $S$ is the set of vertices of $T$ which are not origin of $P$ and the paths $P_1,...,P_k$ are the paths of $T$ whose origins are precisely the vertices of $V(T)\setminus S$.\\
Exc 0: $[3A;(1,1);\{1,2,3\}]$
\\Exc 1: $[4A;(1,1,1);\{1,2,3\};4213]$
\\Exc 2: $[4A;(1,2);\{3,4\};1324;2314]$
\\Exc 3: $[4A;(2,1);\{1,2,\};3421;4132]$
\\Exc 4: $[5A;(1,1,1,1);\{1,2,3,4,5\}]$
\\Exc 5: $[5B;(2,1,1);\{1,2,3\};45213;51423]$
\\Exc 6: $[5C;(1,1,2);\{4,5\};12534;23514;31524]$
\\Exc 7: $[5C;(2,1,1);\{1,2,3,4\};51432]$
\\Exc 8: $[5D;(1,1,1,1);\{2,5\};12543;35124;42153]$
\\Exc 9: $[5E;(1,1,1,1);\{2,4,53\};12453;35421]$
\\Exc 10: $[5E;(1,2,1);\{3,5\};12435;23145;45312]$
\\Exc 11: $[5E;(2,2);\{1,2\};34215;42315;52314]$
\\Exc 12: $[5E;(1,1,2);\{1,2\};35412;41523;51423]$
\\Exc 13: $[6A;(3,1,1);\{3,4\};156324;256143;562341;612345]$
\\Exc 14: $[6B;(2,1,1,1);\{3,4\};154326;254316;562143;612345]$
\\Exc 15: $[6C;(1,1,2,1);\{1,2,3,6\};435261;534261]$
\\Exc 16: $[6C;(1,2,1,1);\{4,5,6\};163425;263415;362415]$
\\Exc 17: $[6D;(2,1,1,1);\{2,4,6\};124365;346521;562143]$
\\Exc 18: $[6D;(1,2,2);\{2,4,6\};126345;341562;564123]$
\\Exc 19: $[6D;(1,1,1,2);\{2,4,6\};126543;341265;563421]$
\\Exc 20: $[6E;(1,1,1,1,1);\{1,2\};341256;465213;516324;621435]$
\\Exc 21: $[6E;(2,1,1,1);\{1,2\};346521;452136;562143;634125]$
\\Exc 21: $[6E;(2,1,1,1);\{1,2\};346521;452136;562143;634125]$
\\Exc 22: $[6F;(1,1,1,1,1);\{1,2,3\};421563;532641;613452]$
\\Exc 23: $[6G;(1,1,1,1,1);\{4,6\};145632;216453;326415;546132]$
\\Exc 24: $[6H;(1,1,1,1,1);\{1,2,3,4\};543162;613425]$
\\Exc 25: $[6H;(1,1,1,2);\{4,5\};142536;243516;341526;613452]$
\\Exc 26: $[6H;(1,1,3);\{4,5,6\};145623;245631;345612]$
\\Exc 27: $[6H;(1,3,1);\{4,6\};126534;236514;316524;543261]$
\\Exc 28: $[6H;(2,1,2);\{4,5\};124563;234561;314562;614235]$
\\Exc 29: $[6I;(1,1,1,1,1);\{4,6\};145632;213654;365421;546231]$
\\Exc 30: $[6J;(1,1,1,1,1);\{4,6\};162453;261453;312465;542631]$
\\Exc 31: $[6K;(1,2,2);\{3,4\};146532;246531;541632;634125]$
\\Exc 32: $[6L;(1,2,1,1);\{5,6\};163425;263415;361425;456132]$
\\Exc 33: $[7A;(1,1,1,1,1,1);\{1,2,3,4,5,6,7\}]$
\\Exc 34: $[7B;(1,1,2,1,1);\{1,2,3\};4576132;5674132;6475132;7541263]$
\\Exc 35: $[7B;(2,1,3);\{1,2,3\};4315627;5316427;6314527;7435612]$
\\Exc 36: $[7B;(2,3,1);\{1,2\};3125476;4567132;5647132;6457132;7421356]$
\\Exc 37: $[7C;(1,1,1,1,1,1);\{4,5,6\};1243567;2341567;3142567;7541632]$
\\Exc 38: $[7C;(1,1, 2,1,1);\{1,2,3\};4156327;5164327;6145327;7541263]$
\\Exc 39: $[7C;(2,1,3);\{1,2,3\};4315627;5316427;6314527;7435612]$
\\Exc 40: $[7D;(1,1,1,2,1);\{1,2\};3412756;4512736;5312746;6215437;7215436]$
\\Exc 41: $[7D;(1,1,1,3);\{6,7\};1546327;2546317;3745216;4753216;5734216]$
\\Exc 42: $[7D;(2,2,1,1);\{6,7\};1342675;2341675;3465127;4563127;5364127]$
\\Exc 43: $[7E;(1,1,2,1,1);\{2,7\};1236745;3214756;4213756;5213746;6734215]$
\\Exc 44: $[7F;(1,1,1,3);\{6,7\};1732546;2713546;3721546;4127635;5127634]$
\\Exc 45: $[7G;(2,1,2,1);\{1,7\};2654317;3654721;4367125;5367124;6517234]$
\\Exc 46: $[7H;(2,2,2);\{4,7\};1746532;2746531;3126574;5321674;6247531]$
\\Exc 47: $[7I;(1,1,2,1,1);\{4,5,7\};1456237;2456137;3456127;6135427]$
\\Exc 48: $[7J;(1,1,2,1,1);\{1,2\};3245167;4235167;5234167;6234157;7234156]$
\\Exc 49: $[8A;(1,1,1,1,1,1,1);\{1,2\};35461278;46527183;56487213;67341285;74358216;\\85347216]$
\\Exc 50: $[8A;(2,1,1,1,1,1);\{1,2\};34652718;46752138;56734128;68214375;78216453;\\83412576]$
\\Exc 51: $[8B;(2,1,2,1,1);\{2,8\};13245867;32145867;42156873;52164873;62145873;\\73245861]$.\\

\begin{figure}[H]
\begin{adjustbox}{max width=\textwidth, max height=\textheight}
\scalebox{0.5}{
\begin{tikzpicture}
\tikzset{enclosed/.style={draw,circle,inner sep=2pt,minimum size=4pt,fill=black}}
\tikzset{->-/.style={decoration={
            markings,
            mark=at position #1 with
            {\arrow{>}}},postaction={decorate}}}
\node[enclosed,label={left,yshift=.2cm:1}](1)at(0,0){};
\node[enclosed,label={left,yshift=.2cm:2}](2)at(1.5,2.6){};
\node[enclosed,label={right,yshift=.2cm:3}](3)at(3,0){};
\draw[black,->-=.5] (1)--(2);
\draw[black,->-=.5] (2)--(3);
\draw[black,->-=.5] (3)--(1);

\node[minimum size=.1pt,label={left:3A}](3A)at(2.15,-.5){};

\node[ellipse,minimum width=4cm,minimum height=1.3cm,draw](4Aa)at(7.6,3){};
\node[minimum size=.1pt](4Aa1)at(6.7,2.52){};
\node[minimum size=.1pt](4Aa2)at(8.5,2.52){};
\node[enclosed,label={left,yshift=.2cm:1}](4A1)at(6.4,3){};
\node[enclosed,label={right,yshift=.2cm:2}](4A2)at(8.8,3){};
\draw[black,->-=.5] (4A1)--(4A2);
\node[enclosed,label={right,yshift=.2cm:3}](4A3)at(9.1,0){};
\node[enclosed,label={left,yshift=.2cm:4}](4A4)at(6.1,0){};
\draw[black,->-=.5] (4A3)--(4A4);
\draw[black,->-=.5] (4A4)--(4Aa1);
\draw[black,->-=.5] (4Aa2)--(4A3);

\node[minimum size=.1pt,label={left:4A}](4A)at(8.25,-.5){};

\node[enclosed,label={left,yshift=.2cm:1}](5A1)at(11.58,2.15){};
\node[enclosed,label={right,yshift=.2cm:2}](5A2)at(13.5,3.8){};
\node[enclosed,label={right,yshift=.2cm:3}](5A3)at(15.42,2.15){};
\node[enclosed,label={right,yshift=.2cm:4}](5A4)at(14.7,0){};\node[enclosed,label={left,yshift=.2cm:5}](5A5)at(12.3,0){};

\draw[black,->-=.5] (5A1)--(5A2);
\draw[black,->-=.5] (5A1)--(5A3);
\draw[black,->-=.5] (5A2)--(5A3);
\draw[black,->-=.5] (5A2)--(5A4);
\draw[black,->-=.5] (5A3)--(5A4);
\draw[black,->-=.5] (5A3)--(5A5);
\draw[black,->-=.5] (5A4)--(5A5);
\draw[black,->-=.5] (5A4)--(5A1);
\draw[black,->-=.5] (5A5)--(5A1);
\draw[black,->-=.5] (5A5)--(5A2);

\node[minimum size=.1pt,label={left:5A}](5A)at(14.15,-.5){};

\node[ellipse,minimum width=4cm,minimum height=4cm,draw](5Ba)at(1.5,-3.2){};
\node[minimum size=.1pt](5Ba1)at(0.6,-4.8){};
\node[minimum size=.1pt](5Ba2)at(2.4,-4.8){};
\node[enclosed,label={left,yshift=.2cm:1}](5B1)at(0.3,-4){};
\node[enclosed,label={right,yshift=.2cm:2}](5B2)at(1.5,-1.93){};\node[enclosed,label={right,yshift=.2cm:3}](5B3)at(2.7,-4){};
\draw[black,->-=.5] (5B1)--(5B2);
\draw[black,->-=.5] (5B2)--(5B3);
\draw[black,->-=.5] (5B1)--(5B3);
\node[enclosed,label={right,yshift=.2cm:4}](5B4)at(3,-7){};
\node[enclosed,label={left,yshift=.2cm:5}](5B5)at(0,-7){};
\draw[black,->-=.5] (5B4)--(5B5);
\draw[black,->-=.5] (5B5)--(5Ba1);
\draw[black,->-=.5] (5Ba2)--(5B4);

\node[minimum size=.1pt,label={left:5B}](5B)at(2.15,-7.5){};

\node[ellipse,minimum width=4cm,minimum height=4cm,draw](5Ca)at(7.6,-3.2){};
\node[minimum size=.1pt](5Ca1)at(6.7,-4.8){};
\node[minimum size=.1pt](5Ca2)at(8.5,-4.8){};
\node[enclosed,label={left,yshift=.2cm:1}](5C1)at(6.4,-4){};
\node[enclosed,label={right,yshift=.2cm:2}](5C2)at(7.6,-1.93){};\node[enclosed,label={right,yshift=.2cm:3}](5C3)at(8.8,-4){};
\draw[black,->-=.5] (5C1)--(5C2);
\draw[black,->-=.5] (5C2)--(5C3);
\draw[black,->-=.5] (5C3)--(5C1);
\node[enclosed,label={right,yshift=.2cm:4}](5C4)at(9.1,-7){};
\node[enclosed,label={left,yshift=.2cm:5}](5C5)at(6.1,-7){};
\draw[black,->-=.5] (5C4)--(5C5);
\draw[black,->-=.5] (5C5)--(5Ca1);
\draw[black,->-=.5] (5Ca2)--(5C4);

\node[minimum size=.1pt,label={left:5C}](5C)at(8.25,-7.5){};

\node[ellipse,minimum width=4cm,minimum height=1.3cm,draw](5D)at(13.5,-4){};
\node[minimum size=.1pt](5Da1)at(12.6,-4.48){};
\node[minimum size=.1pt](5Da2)at(14.4,-4.48){};
\node[enclosed,label={left,yshift=.2cm:1}](5D1)at(12.3,-4){};
\node[enclosed,label={right,yshift=.2cm:2}](5D2)at(14.7,-4){};
\draw[black,->-=.5] (5D1)--(5D2);
\node[enclosed,label={right,yshift=.2cm:3}](5D3)at(15,-7){};
\node[enclosed,label={left,yshift=.2cm:4}](5D4)at(12,-7){};
\draw[black,->-=.5] (5D3)--(5D4);
\draw[black,->-=.5] (5D4)--(5Da1);
\draw[black,->-=.5] (5Da2)--(5D3);
\node[enclosed,label={left,yshift=.01cm:5}](5D5)at(13.5,-5.5){};
\draw[black,->-=.5] (5D1)--(5D5);
\draw[black,->-=.5] (5D3)--(5D5);
\draw[black,->-=.5] (5D5)--(5D2);
\draw[black,->-=.5] (5D5)--(5D4);

\node[minimum size=.1pt,label={left:5D}](5D)at(14.15,-7.5){};

\node[ellipse,minimum width=4cm,minimum height=1.3cm,draw](5Ea)at(2.4,-9.5){};
\node[minimum size=.1pt](5Ea1)at(1.5,-9.98){};
\node[minimum size=.1pt](5Ea2)at(3.3,-9.98){};
\node[enclosed,label={left,yshift=.2cm:1}](5E1)at(1.2,-9.5){};
\node[enclosed,label={right,yshift=.2cm:2}](5E2)at(3.6,-9.5){};
\draw[black,->-=.5] (5E1)--(5E2);
\node[enclosed,label={right,yshift=.2cm:3}](5E3)at(4.4,-12.5){};
\node[ellipse,minimum width=4cm,minimum height=1.3cm,draw](5Eb)at(1.25,-12.5){};
\node[enclosed,label={left,yshift=.2cm:4}](5E4)at(0.05,-12.5){};
\node[enclosed,label={right,yshift=.2cm:5}](5E5)at(2.45,-12.5){};
\draw[black,->-=.5] (5E3)--(5Eb);
\draw[black,->-=.5] (5Ea)--(5E3);
\draw[black,->-=.5] (5E4)--(5E5);
\draw[black,->-=.5] (5Eb)--(5Ea);

\node[minimum size=.1pt,label={left:5E}](5E)at(3.1,-13.65){};

\node[ellipse,minimum width=5cm,minimum height=3.2cm,draw](6Aa)at(8.1,-9.95){};
\node[minimum size=.1pt](6Aa1)at(6.5,-11.07){};
\node[minimum size=.1pt](6Aa2)at(9.7,-11.07){};
\node[ellipse,minimum width=2.2cm,minimum height=1.3cm,draw](6Ab)at(9.1,-10.5){};
\node[enclosed,label={left,yshift=.2cm:1}](6A1)at(6.9,-10.5){};
\node[enclosed,label={right,yshift=.2cm:2}](6A2)at(8.1,-9){};
\node[enclosed,label={left,yshift=.01cm:3}](6A3)at(8.55,-10.5){};
\node[enclosed,label={right,yshift=.01cm:4}](6A4)at(9.65,-10.5){};\node[enclosed,label={right,yshift=.2cm:5}](6A5)at(9.7,-12.5){};\node[enclosed,label={left,yshift=.2cm:6}](6A6)at(6.5,-12.5){};

\draw[black,->-=.5] (6A1)--(6A2);
\draw[black,->-=.5] (6A2)--(6Ab);
\draw[black,->-=.5] (6Ab)--(6A1);
\draw[black,->-=.5] (6A4)--(6A3);
\draw[black,->-=.5] (6Aa2)--(6A5);
\draw[black,->-=.5] (6A5)--(6A6);
\draw[black,->-=.5] (6A6)--(6Aa1);

\node[minimum size=.1pt,label={left:6A}](6A)at(8.75,-13.65){};

\node[enclosed,label={left,yshift=.2cm:6}](6B6)at(12.58,-10.35){};

\node[ellipse,minimum width=2.5cm,minimum height=1.2cm,draw](6Ba)at(14.5,-8.8){};

\node[enclosed,label={right,yshift=.3cm:1}](6B1)at(13.7,-8.8){};
\node[enclosed,label={left,yshift=.3cm:2}](6B2)at(15.3,-8.8){};
\node[enclosed,label={right,yshift=.2cm:3}](6B3)at(16.42,-10.35){};
\node[enclosed,label={right,yshift=.2cm:4}](6B4)at(15.7,-12.5){};\node[enclosed,label={left,yshift=.2cm:5}](6B5)at(13.3,-12.5){};
\draw[black,->-=.5] (6B1)--(6B2);
\draw[black,->-=.5] (6B6)--(6Ba);
\draw[black,->-=.5] (6B6)--(6B3);
\draw[black,->-=.5] (6Ba)--(6B5);
\draw[black,->-=.5] (6Ba)--(6B4);
\draw[black,->-=.5] (6B3)--(6B4);
\draw[black,->-=.5] (6B3)--(6Ba);
\draw[black,->-=.5] (6B5)--(6B4);
\draw[black,->-=.5] (6B4)--(6B6);
\draw[black,->-=.5] (6B5)--(6B6);
\draw[black,->-=.5] (6B5)--(6B3);

\node[minimum size=.1pt,label={left:6B}](6B)at(15.15,-13.65){};

\end{tikzpicture}}
\scalebox{0.5}{
\begin{tikzpicture}
\tikzset{enclosed/.style={draw,circle,inner sep=2pt,minimum size=4pt,fill=black}}
\tikzset{->-/.style={decoration={
            markings,
            mark=at position #1 with
            {\arrow{>}}},postaction={decorate}}}
            
\node[ellipse,minimum width=3cm,minimum height=3cm,draw](a)at(0,0){};
\node[enclosed,label={left,yshift=.2cm:1}](1)at(-0.9,-0.7){};
\node[enclosed,label={right,yshift=.2cm:2}](2)at(0,0.9){};\node[enclosed,label={right,yshift=.2cm:3}](3)at(0.9,-0.7){};
\draw[black,->-=.5] (1)--(2);
\draw[black,->-=.5] (2)--(3);
\draw[black,->-=.5] (3)--(1);
\node[enclosed,label={right,yshift=.2cm:6}](6)at(1.5,-3.3){};
\node[ellipse,minimum width=3cm,minimum height=1.3cm,draw](b)at(-1.5,-3.3){};
\node[enclosed,label={right,yshift=.3cm:4}](4)at(-2.5,-3.3){};
\node[enclosed,label={left,yshift=.3cm:5}](5)at(-0.5,-3.3){};
\draw[black,->-=.5] (4)--(5);
\draw[black,->-=.5] (6)--(b);
\draw[black,->-=.5] (b)--(a);
\draw[black,->-=.5] (a)--(6);

\node[minimum size=.1pt,label={left:6C}](6C)at(.65,-4){};

\node[ellipse,minimum width=3cm,minimum height=1.3cm,draw](6Da)at(5,0.3){};
\node[ellipse,minimum width=1.3cm,minimum height=3cm,draw](6Db)at(8,-1.8){};
\node[ellipse,minimum width=3cm,minimum height=1.3cm,draw](6Dc)at(5,-3.3){};
\node[enclosed,label={right,yshift=.3cm:1}](6D1)at(4,0.3){};
\node[enclosed,label={left,yshift=.3cm:2}](6D2)at(6,0.3){};
\node[enclosed,label={left,yshift=-.2cm:3}](6D3)at(8,-0.8){};
\node[enclosed,label={left,yshift=.2cm:4}](6D4)at(8,-2.8){};
\node[enclosed,label={left,yshift=.3cm:5}](6D5)at(6,-3.3){};
\node[enclosed,label={right,yshift=.3cm:6}](6D6)at(4,-3.3){};

\draw[black,->-=.5] (6D1)--(6D2);
\draw[black,->-=.5] (6D3)--(6D4);
\draw[black,->-=.5] (6D5)--(6D6);
\draw[black,->-=.5] (6Da)--(6Db);
\draw[black,->-=.5] (6Db)--(6Dc);
\draw[black,->-=.5] (6Dc)--(6Da);

\node[minimum size=.1pt,label={left:6D}](6D)at(6.65,-4){};

\node[enclosed,label={left,yshift=.2cm:6}](6E6)at(10.08,-1.15){};

\node[ellipse,minimum width=2.5cm,minimum height=1.2cm,draw](6Ea)at(12,0.4){};

\node[enclosed,label={right,yshift=.3cm:1}](6E1)at(11.2,0.4){};
\node[enclosed,label={left,yshift=.3cm:2}](6E2)at(12.8,0.4){};

\node[enclosed,label={right,yshift=.2cm:3}](6E3)at(13.92,-1.15){};
\node[enclosed,label={right,yshift=.2cm:4}](6E4)at(13.2,-3.3){};\node[enclosed,label={left,yshift=.2cm:5}](6E5)at(10.8,-3.3){};
\draw[black,->-=.5] (6E1)--(6E2);
\draw[black,->-=.5] (6E6)--(6Ea);
\draw[black,->-=.5] (6E6)--(6E3);
\draw[black,->-=.5] (6Ea)--(6E3);
\draw[black,->-=.5] (6Ea)--(6E4);
\draw[black,->-=.5] (6E3)--(6E4);
\draw[black,->-=.5] (6E3)--(6E5);
\draw[black,->-=.5] (6E4)--(6E5);
\draw[black,->-=.5] (6E4)--(6E6);
\draw[black,->-=.5] (6E5)--(6E6);
\draw[black,->-=.5] (6E5)--(6Ea);

\node[minimum size=.1pt,label={left:6E}](6E)at(12.65,-4){};

\node[enclosed,label={right,yshift=.2cm:1}](6F1)at(1.2,-5.5){};
\node[enclosed,label={right,yshift=.2cm:2}](6F2)at(1.2,-9.9){};
\node[enclosed,label={left,yshift=.2cm:3}](6F3)at(-2.5,-7.7){};
\node[enclosed,label={left,yshift=.2cm:4}](6F4)at(-1.2,-9.9){};
\node[enclosed,label={left,yshift=.2cm:5}](6F5)at(-1.2,-5.5){};\node[enclosed,label={right,yshift=.2cm:6}](6F6)at(2.5,-7.7){};

\draw[black,->-=.4] (6F1)--(6F2);
\draw[black,->-=.5] (6F1)--(6F5);
\draw[black,->-=.2] (6F2)--(6F3);
\draw[black,->-=.5] (6F2)--(6F6);
\draw[black,->-=.2] (6F3)--(6F1);
\draw[black,->-=.5] (6F3)--(6F4);
\draw[black,->-=.4] (6F4)--(6F1);
\draw[black,->-=.5] (6F4)--(6F2);
\draw[black,->-=.2] (6F4)--(6F6);
\draw[black,->-=.4] (6F5)--(6F2);
\draw[black,->-=.5] (6F5)--(6F3);
\draw[black,->-=.4] (6F5)--(6F4);
\draw[black,->-=.5] (6F6)--(6F1);
\draw[black,->-=.4] (6F6)--(6F3);
\draw[black,->-=.2] (6F6)--(6F5);

\node[minimum size=.1pt,label={left:6F}](6F)at(.65,-10.6){};

\node[enclosed,label={right,yshift=.2cm:2}](6G2)at(7.6,-5.5){};
\node[enclosed,label={right,yshift=.2cm:4}](6G4)at(7.6,-9.9){};
\node[enclosed,label={left,yshift=.2cm:6}](6G6)at(3.9,-7.7){};
\node[enclosed,label={left,yshift=.2cm:5}](6G5)at(5.2,-9.9){};
\node[enclosed,label={left,yshift=.2cm:1}](6G1)at(5.2,-5.5){};\node[enclosed,label={right,yshift=.2cm:3}](6G3)at(8.9,-7.7){};

\draw[black,->-=.4] (6G1)--(6G4);
\draw[black,->-=.4] (6G1)--(6G5);
\draw[black,->-=.5] (6G2)--(6G1);
\draw[black,->-=.4] (6G2)--(6G5);
\draw[black,->-=.2] (6G3)--(6G1);
\draw[black,->-=.5] (6G3)--(6G2);
\draw[black,->-=.4] (6G3)--(6G6);
\draw[black,->-=.4] (6G4)--(6G2);
\draw[black,->-=.5] (6G4)--(6G3);
\draw[black,->-=.2] (6G5)--(6G3);
\draw[black,->-=.5] (6G5)--(6G4);
\draw[black,->-=.5] (6G5)--(6G6);
\draw[black,->-=.5] (6G6)--(6G1);
\draw[black,->-=.2] (6G6)--(6G2);
\draw[black,->-=.2] (6G6)--(6G4);

\node[minimum size=.1pt,label={left:6G}](6G)at(7.05,-10.6){};

\node[ellipse,minimum width=3cm,minimum height=3cm,draw](6Ha)at(12,-6.6){};
\node[enclosed,label={left,yshift=.2cm:1}](6H1)at(11.1,-7.3){};
\node[enclosed,label={right,yshift=.2cm:2}](6H2)at(12,-5.7){};\node[enclosed,label={right,yshift=.2cm:3}](6H3)at(12.9,-7.3){};
\draw[black,->-=.5] (6H1)--(6H2);
\draw[black,->-=.5] (6H2)--(6H3);
\draw[black,->-=.5] (6H3)--(6H1);
\node[enclosed,label={right,yshift=.2cm:6}](6H6)at(13.5,-9.9){};
\node[ellipse,minimum width=3cm,minimum height=1.3cm,draw](6Hb)at(10.5,-9.9){};
\node[enclosed,label={right,yshift=.3cm:4}](6H4)at(9.5,-9.9){};
\node[enclosed,label={left,yshift=.3cm:5}](6H5)at(11.5,-9.9){};
\draw[black,->-=.5] (6H5)--(6H4);
\draw[black,->-=.5] (6Hb)--(6H6);
\draw[black,->-=.5] (6Ha)--(6Hb);
\draw[black,->-=.5] (6H6)--(6Ha);

\node[minimum size=.1pt,label={left:6H}](6H)at(12.5,-10.6){};

\node[ellipse,minimum width=4cm,minimum height=1.3cm,draw](6Ia)at(0.2,-13){};
\node[minimum size=.1pt](6Ia1)at(-0.7,-13.48){};
\node[minimum size=.1pt](6Ia2)at(1.1,-13.48){};
\node[enclosed,label={left,yshift=.1cm:5}](6I5)at(-1,-13){};
\node[enclosed,label={right,yshift=.1cm:4}](6I4)at(1.4,-13){};
\draw[black,->-=.5] (6I5)--(6I4);
\node[enclosed,label={right,yshift=.2cm:3}](6I3)at(2.2,-16){};

\node[ellipse,minimum width=4cm,minimum height=1.3cm,draw](6Ib)at(-1.3,-16){};
\node[enclosed,label={left,yshift=.1cm:2}](6I2)at(-2.5,-16){};
\node[enclosed,label={right,yshift=.1cm:1}](6I1)at(-0.1,-16){};
\draw[black,->-=.5] (6I2)--(6I1);
\draw[black,->-=.5] (6I3)--(6Ib);
\draw[black,->-=.5] (6Ib)--(6Ia1);
\draw[black,->-=.5] (6Ia2)--(6I3);
\node[enclosed,label={right,yshift=.1cm:6}](6I6)at(0.2,-14.8){};
\draw[black,->-=.5] (6I5)--(6I6);
\draw[black,->-=.5] (6I6)--(6I4);
\draw[black,->-=.5] (6I3)--(6I6);
\draw[black,->-=.7] (6I1)--(6I6);
\draw[black,->-=.7] (6I6)--(6I2);

\node[minimum size=.1pt,label={left:6I}](6I)at(0.5,-16.7){};

\node[ellipse,minimum width=3cm,minimum height=3cm,draw](6Ja)at(6,-12.7){};
\node[enclosed,label={left,yshift=.2cm:2}](6J2)at(5.1,-13.4){};
\node[enclosed,label={right,yshift=.2cm:1}](6J1)at(6,-11.8){};\node[enclosed,label={right,yshift=.2cm:6}](6J6)at(6.9,-13.4){};
\draw[black,->-=.5] (6J2)--(6J1);
\draw[black,->-=.5] (6J2)--(6J6);
\draw[black,->-=.5] (6J1)--(6J6);
\node[enclosed,label={right,yshift=.2cm:3}](6J3)at(7.5,-16){};
\node[ellipse,minimum width=3cm,minimum height=1.3cm,draw](6Jb)at(4.5,-16){};
\node[enclosed,label={right,yshift=.3cm:4}](6J4)at(3.5,-16){};
\node[enclosed,label={left,yshift=.3cm:5}](6J5)at(5.5,-16){};
\draw[black,->-=.5] (6J5)--(6J4);
\draw[black,->-=.5] (6Jb)--(6J3);
\draw[black,->-=.5] (6Ja)--(6Jb);
\draw[black,->-=.5] (6J3)--(6Ja);

\node[minimum size=.1pt,label={left:6J}](6J)at(6.2,-16.7){};

\node[enclosed,label={left,yshift=.2cm:6}](6K6)at(10.08,-13.85){};

\node[ellipse,minimum width=2.5cm,minimum height=1.2cm,draw](6Ka)at(12,-12.3){};

\node[enclosed,label={right,yshift=.2cm:1}](6K1)at(11.08,-12.3){};
\node[enclosed,label={left,yshift=.2cm:2}](6K2)at(12.88,-12.3){};

\node[enclosed,label={right,yshift=.2cm:3}](6K3)at(13.92,-13.85){};
\node[enclosed,label={right,yshift=.2cm:4}](6K4)at(13.2,-16){};\node[enclosed,label={left,yshift=.2cm:5}](6K5)at(10.8,-16){};
\draw[black,->-=.5] (6K1)--(6K2);
\draw[black,->-=.5] (6K6)--(6Ka);
\draw[black,->-=.5] (6K6)--(6K3);
\draw[black,->-=.5] (6K3)--(6Ka);
\draw[black,->-=.5] (6Ka)--(6K4);
\draw[black,->-=.5] (6K4)--(6K3);
\draw[black,->-=.5] (6K5)--(6K3);
\draw[black,->-=.5] (6K5)--(6K4);
\draw[black,->-=.5] (6K6)--(6K4);
\draw[black,->-=.5] (6K5)--(6K6);
\draw[black,->-=.5] (6Ka)--(6K5);

\node[minimum size=.1pt,label={left:6K}](6K)at(12.65,-16.7){};

\end{tikzpicture}}
\end{adjustbox}
\caption{Finite exceptions -1-}
\label{figure 1}
\end{figure}
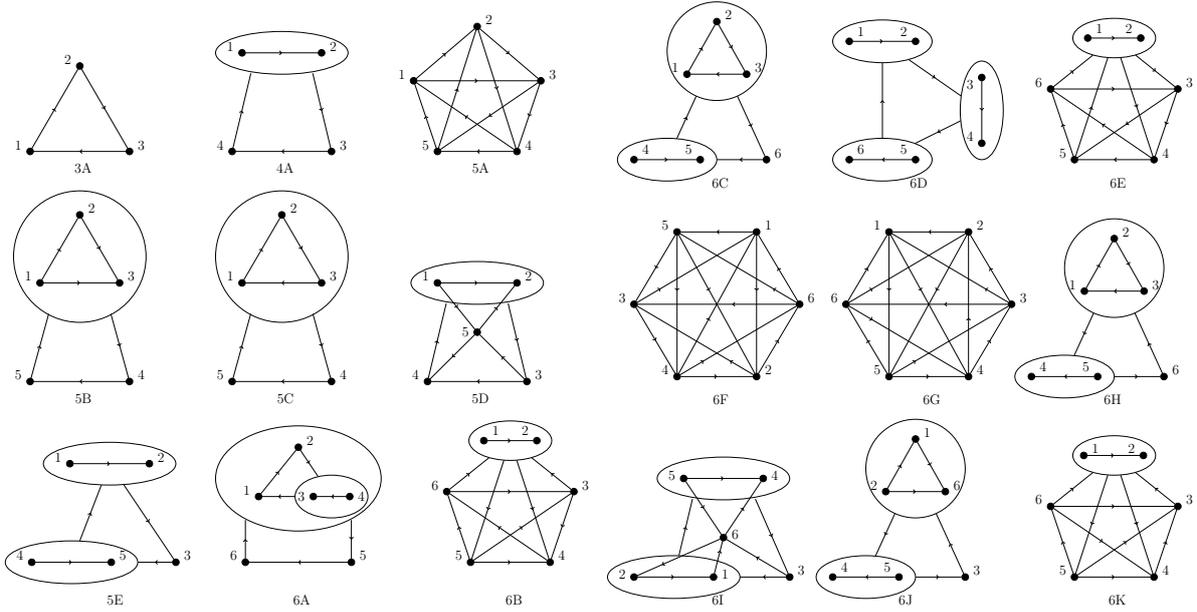
\begin{figure}
\begin{adjustbox}{max width=\textwidth, max height=\textheight}
\scalebox{0.5}{
\begin{tikzpicture}
\tikzset{enclosed/.style={draw,circle,inner sep=2pt,minimum size=4pt,fill=black}}
\tikzset{->-/.style={decoration={
            markings,
            mark=at position #1 with
            {\arrow{>}}},postaction={decorate}}}
            
\node[ellipse,minimum width=3cm,minimum height=3cm,draw](a)at(0,0){};
\node[enclosed,label={left,yshift=.2cm:1}](6L1)at(-0.9,-0.7){};
\node[enclosed,label={right,yshift=.2cm:2}](6L2)at(0,0.9){};\node[enclosed,label={right,yshift=.2cm:3}](6L3)at(0.9,-0.7){};
\draw[black,->-=.5] (6L1)--(6L2);
\draw[black,->-=.5] (6L1)--(6L3);
\draw[black,->-=.5] (6L2)--(6L3);
\node[enclosed,label={right,yshift=.2cm:6}](6L6)at(1.5,-3.3){};
\node[ellipse,minimum width=3cm,minimum height=1.3cm,draw](b)at(-1.5,-3.3){};
\node[enclosed,label={right,yshift=.3cm:5}](6L5)at(-2.5,-3.3){};
\node[enclosed,label={left,yshift=.3cm:4}](6L4)at(-0.5,-3.3){};
\draw[black,->-=.5] (6L4)--(6L5);
\draw[black,->-=.5] (6L6)--(b);
\draw[black,->-=.5] (b)--(a);
\draw[black,->-=.5] (a)--(6L6);

\node[minimum size=.1pt,label={left:6L}](6L)at(.65,-4){};

\node[enclosed,label={right,yshift=.2cm:1}](1)at(6.2,1){};
\node[enclosed,label={right,yshift=.2cm:2}](2)at(7.6,-.7){};
\node[enclosed,label={right,yshift=.2cm:3}](3)at(7,-2.5){};
\node[enclosed,label={right,yshift=-.2cm:4}](4)at(5.3,-3.3){};
\node[enclosed,label={left,yshift=.2cm:5}](5)at(3.6,-2.5){};
\node[enclosed,label={left,yshift=.2cm:6}](6)at(3,-.7){};
\node[enclosed,label={left,yshift=.2cm:7}](7)at(4.4,1){};

\draw[black,->-=.5] (1)--(2);
\draw[black,->-=.4] (1)--(3);
\draw[black,->-=.65] (1)--(5);
\draw[black,->-=.5] (2)--(3);
\draw[black,->-=.3] (2)--(4);
\draw[black,->-=.7] (2)--(6);
\draw[black,->-=.5] (3)--(4);
\draw[black,->-=.4] (3)--(5);
\draw[black,->-=.6] (3)--(7);
\draw[black,->-=.5] (4)--(5);
\draw[black,->-=.4] (4)--(6);
\draw[black,->-=.7] (4)--(1);
\draw[black,->-=.5] (5)--(6);
\draw[black,->-=.35] (5)--(7);
\draw[black,->-=.7] (5)--(2);
\draw[black,->-=.5] (6)--(7);
\draw[black,->-=.2] (6)--(1);
\draw[black,->-=.75] (6)--(3);
\draw[black,->-=.5] (7)--(1);
\draw[black,->-=.2] (7)--(2);
\draw[black,->-=.3] (7)--(4);

\node[minimum size=.1pt,label={left:7A}](7A)at(5.95,-4){};

\node[ellipse,minimum width=3cm,minimum height=3cm,draw](7Ba)at(11.4,0){};
\node[enclosed,label={left,yshift=.2cm:3}](7B3)at(10.6,-0.6){};
\node[enclosed,label={right,yshift=.2cm:1}](7B1)at(11.4,0.8){};\node[enclosed,label={right,yshift=.2cm:2}](7B2)at(12.2,-0.6){};
\draw[black,->-=.5] (7B1)--(7B2);
\draw[black,->-=.5] (7B3)--(7B1);
\draw[black,->-=.5] (7B3)--(7B2);
\node[enclosed,label={right,yshift=.2cm:7}](7B7)at(12.9,-3.3){};
\node[ellipse,minimum width=3cm,minimum height=3cm,draw](7Bb)at(9.9,-3.3){};
\node[enclosed,label={left,yshift=.2cm:6}](7B6)at(9.1,-3.9){};
\node[enclosed,label={right,yshift=.2cm:4}](7B4)at(9.9,-2.5){};
\node[enclosed,label={right,yshift=.2cm:5}](7B5)at(10.7,-3.9){};
\draw[black,->-=.5] (7B4)--(7B5);
\draw[black,->-=.5] (7B5)--(7B6);
\draw[black,->-=.5] (7B6)--(7B4);
\draw[black,->-=.5] (7Ba)--(7B7);
\draw[black,->-=.5] (7B7)--(7Bb);
\draw[black,->-=.5] (7Bb)--(7Ba);

\node[minimum size=.1pt,label={left:7B}](7B)at(12.35,-4){};

\node[ellipse,minimum width=3cm,minimum height=3cm,draw](7Ca)at(2.8,-6.8){};
\node[enclosed,label={left,yshift=.2cm:3}](7C3)at(2,-7.4){};
\node[enclosed,label={right,yshift=.2cm:1}](7C1)at(2.8,-6){};\node[enclosed,label={right,yshift=.2cm:2}](7C2)at(3.6,-7.4){};
\draw[black,->-=.5] (7C1)--(7C2);
\draw[black,->-=.5] (7C3)--(7C1);
\draw[black,->-=.5] (7C2)--(7C3);
\node[enclosed,label={right,yshift=.2cm:7}](7C7)at(4.3,-10.1){};
\node[ellipse,minimum width=3cm,minimum height=3cm,draw](7Cb)at(1.3,-10.1){};
\node[enclosed,label={left,yshift=.2cm:6}](7C6)at(0.5,-10.7){};
\node[enclosed,label={right,yshift=.2cm:4}](7C4)at(1.3,-9.3){};
\node[enclosed,label={right,yshift=.2cm:5}](7C5)at(2.1,-10.7){};
\draw[black,->-=.5] (7C4)--(7C5);
\draw[black,->-=.5] (7C5)--(7C6);
\draw[black,->-=.5] (7C6)--(7C4);
\draw[black,->-=.5] (7Ca)--(7C7);
\draw[black,->-=.5] (7C7)--(7Cb);
\draw[black,->-=.5] (7Cb)--(7Ca);

\node[minimum size=.1pt,label={left:7C}](7C)at(3.5,-11.2){};

\node[ellipse,minimum width=3cm,minimum height=1.5cm,draw](7Da)at(6.5,-7.1){};
\node[ellipse,minimum width=3cm,minimum height=3cm,draw](7Db)at(10.5,-8.7){};
\node[ellipse,minimum width=3cm,minimum height=1.5cm,draw](7Dc)at(6.5,-10.1){};
\node[enclosed,label={right,yshift=.3cm:1}](7D1)at(5.5,-7.1){};
\node[enclosed,label={left,yshift=.3cm:2}](7D2)at(7.5,-7.1){};
\node[enclosed,label={left,yshift=.2cm:3}](7D3)at(10.5,-7.8){};
\node[enclosed,label={left,yshift=.2cm:5}](7D5)at(9.8,-9.1){};
\node[enclosed,label={right,yshift=.2cm:4}](7D4)at(11.2,-9.1){};
\node[enclosed,label={left,yshift=.3cm:6}](7D6)at(7.5,-10.1){};
\node[enclosed,label={right,yshift=.3cm:7}](7D7)at(5.5,-10.1){};

\draw[black,->-=.5] (7D1)--(7D2);
\draw[black,->-=.5] (7D3)--(7D4);
\draw[black,->-=.5] (7D4)--(7D5);
\draw[black,->-=.5] (7D5)--(7D3);
\draw[black,->-=.5] (7D6)--(7D7);
\draw[black,->-=.5] (7Da)--(7Db);
\draw[black,->-=.5] (7Db)--(7Dc);
\draw[black,->-=.5] (7Dc)--(7Da);

\node[minimum size=.1pt,label={left:7D}](7D)at(8.5,-11){};

\node[ellipse,minimum width=5cm,minimum height=4cm,draw](7Ea)at(3.5,-14.5){};
\node[ellipse,minimum width=3cm,minimum height=3cm,draw](7Eb)at(2.7,-14.5){};
\node[ellipse,minimum width=3cm,minimum height=1.5cm,draw](7Ec)at(-1.5,-16){};
\node[enclosed,label={right,yshift=.3cm:1}](7E1)at(-1.5,-13){};
\node[enclosed,label={right,yshift=.3cm:2}](7E2)at(5.3,-14.5){};
\node[enclosed,label={left,yshift=.2cm:3}](7E3)at(2.5,-13.7){};
\node[enclosed,label={left,yshift=.2cm:5}](7E5)at(1.8,-15){};
\node[enclosed,label={right,yshift=.2cm:4}](7E4)at(3.2,-15){};
\node[enclosed,label={left,yshift=.3cm:6}](7E6)at(-0.5,-16){};
\node[enclosed,label={right,yshift=.3cm:7}](7E7)at(-2.5,-16){};

\draw[black,->-=.5] (7E1)--(7Ea);
\draw[black,->-=.5] (7E3)--(7E4);
\draw[black,->-=.5] (7E4)--(7E5);
\draw[black,->-=.5] (7E5)--(7E3);
\draw[black,->-=.5] (7E6)--(7E7);
\draw[black,->-=.5] (7Eb)--(7E2);
\draw[black,->-=.5] (7Ea)--(7Ec);
\draw[black,->-=.5] (7Ec)--(7E1);

\node[minimum size=.1pt,label={left:7E}](7E)at(.65,-17){};

\node[ellipse,minimum width=3cm,minimum height=1.5cm,draw](7Fa)at(8,-13){};
\node[ellipse,minimum width=3cm,minimum height=3cm,draw](7Fb)at(12,-14.5){};
\node[ellipse,minimum width=3cm,minimum height=1.5cm,draw](7Fc)at(8,-16){};
\node[enclosed,label={right,yshift=.3cm:5}](7F5)at(7,-13){};
\node[enclosed,label={left,yshift=.3cm:4}](7F4)at(9,-13){};
\node[enclosed,label={left,yshift=.2cm:1}](7F1)at(12,-13.7){};
\node[enclosed,label={left,yshift=.2cm:2}](7F2)at(11.3,-15){};
\node[enclosed,label={right,yshift=.2cm:3}](7F3)at(12.7,-15){};
\node[enclosed,label={left,yshift=.3cm:6}](7F6)at(9,-16){};
\node[enclosed,label={right,yshift=.3cm:7}](7F7)at(7,-16){};

\draw[black,->-=.5] (7F1)--(7F3);
\draw[black,->-=.5] (7F3)--(7F2);
\draw[black,->-=.5] (7F2)--(7F1);
\draw[black,->-=.5] (7F4)--(7F5);
\draw[black,->-=.5] (7F6)--(7F7);
\draw[black,->-=.5] (7Fa)--(7Fb);
\draw[black,->-=.5] (7Fb)--(7Fc);
\draw[black,->-=.5] (7Fc)--(7Fa);
\draw[black,->-=.5] (7F5)--(7F7);

\node[minimum size=.1pt,label={left:7F}](7F)at(10,-17){};

\end{tikzpicture}}
\scalebox{0.5}{
\begin{tikzpicture}
\tikzset{enclosed/.style={draw,circle,inner sep=2pt,minimum size=4pt,fill=black}}
\tikzset{->-/.style={decoration={
            markings,
            mark=at position #1 with
            {\arrow{>}}},postaction={decorate}}}
            
\node[ellipse,minimum width=5cm,minimum height=4cm,draw](a)at(5,-1.5){};
\node[ellipse,minimum width=1.3cm,minimum height=3cm,draw](b)at(6,-1.5){};
\node[ellipse,minimum width=3cm,minimum height=1.5cm,draw](c)at(0,-3){};
\node[enclosed,label={right,yshift=.3cm:1}](1)at(6,-2.3){};
\node[enclosed,label={left,yshift=.2cm:2}](2)at(4,-2.3){};
\node[enclosed,label={left,yshift=.2cm:3}](3)at(4,-0.7){};
\node[enclosed,label={left,yshift=.2cm:4}](4)at(1,-3){};
\node[enclosed,label={right,yshift=.2cm:5}](5)at(-1,-3){};
\node[enclosed,label={left,yshift=.2cm:6}](6)at(0,0){};
\node[enclosed,label={right,yshift=.2cm:7}](7)at(6,-0.7){};
\node[minimum size=.1pt](b1)at(5.63,-0.7){};
\node[minimum size=.1pt](b2)at(5.63,-2.3){};

\draw[black,->-=.5] (2)--(3);
\draw[black,->-=.5] (3)--(b1);
\draw[black,->-=.5] (b2)--(2);
\draw[black,->-=.5] (7)--(1);
\draw[black,->-=.5] (a)--(6);
\draw[black,->-=.5] (6)--(c);
\draw[black,->-=.5] (c)--(a);
\draw[black,->-=.5] (4)--(5);

\node[minimum size=.1pt,label={left:7G}](7G)at(2.65,-4){};

\node[enclosed,label={left,yshift=.2cm:6}](7H6)at(9.2,-.85){};

\node[ellipse,minimum width=2.5cm,minimum height=1.2cm,draw](7Ha)at(11.12,.7){};

\node[enclosed,label={right,yshift=.3cm:1}](7H1)at(10.2,.7){};
\node[enclosed,label={left,yshift=.3cm:2}](7H2)at(12,.7){};

\node[enclosed,label={right,yshift=.2cm:3}](7H3)at(13.04,-.85){};

\node[ellipse,minimum width=2.5cm,minimum height=1.2cm,draw](7Hb)at(12.32,-3){};

\node[enclosed,label={left,yshift=.3cm:4}](7H4)at(13.22,-3){};
\node[enclosed,label={right,yshift=.3cm:7}](7H7)at(11.42,-3){};

\node[enclosed,label={left,yshift=.2cm:5}](7H5)at(9.92,-3){};

\draw[black,->-=.5] (7H1)--(7H2);
\draw[black,->-=.5] (7H6)--(7Ha);
\draw[black,->-=.5] (7H6)--(7H3);
\draw[black,->-=.5] (7H3)--(7Ha);
\draw[black,->-=.5] (7Ha)--(7Hb);
\draw[black,->-=.5] (7Hb)--(7H3);
\draw[black,->-=.5] (7H5)--(7H3);
\draw[black,->-=.5] (7H5)--(7Hb);
\draw[black,->-=.6] (7H6)--(7Hb);
\draw[black,->-=.5] (7H5)--(7H6);
\draw[black,->-=.5] (7Ha)--(7H5);
\draw[black,->-=.5] (7H7)--(7H4);

\node[minimum size=.1pt,label={left:7H}](7H)at(11.72,-4){};

\node[ellipse,minimum width=3cm,minimum height=3cm,draw](7Ia)at(2,-6){};
\node[enclosed,label={left,yshift=.2cm:5}](7I5)at(1.2,-6.6){};
\node[enclosed,label={right,yshift=.2cm:4}](7I4)at(2,-5.2){};\node[enclosed,label={right,yshift=.2cm:7}](7I7)at(2.8,-6.6){};
\node[enclosed,label={right,yshift=.2cm:6}](7I6)at(3.5,-9.3){};
\node[ellipse,minimum width=3cm,minimum height=3cm,draw](7Ib)at(.5,-9.3){};
\node[enclosed,label={left,yshift=.2cm:3}](7I3)at(-.3,-9.9){};
\node[enclosed,label={right,yshift=.2cm:2}](7I2)at(0.5,-8.3){};
\node[enclosed,label={right,yshift=.2cm:1}](7I1)at(1.3,-9.9){};
\draw[black,->-=.5] (7I4)--(7I7);
\draw[black,->-=.5] (7I5)--(7I4);
\draw[black,->-=.5] (7I7)--(7I5);
\draw[black,->-=.5] (7Ia)--(7I6);
\draw[black,->-=.5] (7I6)--(7Ib);
\draw[black,->-=.5] (7Ib)--(7Ia);
\draw[black,->-=.5] (7I3)--(7I2);
\draw[black,->-=.5] (7I2)--(7I1);
\draw[black,->-=.5] (7I3)--(7I1);

\node[minimum size=.1pt,label={left:7I}](7I)at(2.65,-10.5){};

\node[ellipse,minimum width=4cm,minimum height=6cm,draw](7Ja)at(10.5,-7.3){};
\node[ellipse,minimum width=3cm,minimum height=3cm,draw](7Jb)at(10.5,-6.3){};
\node[enclosed,label={left,yshift=.2cm:1}](7J1)at(6.2,-7.9){};
\node[enclosed,label={left,yshift=.2cm:2}](7J2)at(6.2,-6.7){};\node[enclosed,label={right,yshift=.2cm:3}](7J3)at(7.4,-7.2){};
\node[enclosed,label={left,yshift=.2cm:4}](7J4)at(10,-5.7){};
\node[enclosed,label={right,yshift=.2cm:5}](7J5)at(11,-6.2){};
\node[enclosed,label={left,yshift=-.2cm:6}](7J6)at(10,-6.7){};
\node[enclosed,label={right,yshift=.2cm:7}](7J7)at(10.5,-8.9){};
\node[minimum size=.1pt](7Ja1)at(8.9,-5.9){};
\node[minimum size=.1pt](7Ja2)at(8.9,-8.9){};

\draw[black,->-=.5] (7J2)--(7J1);
\draw[black,->-=.5] (7J1)--(7J3);
\draw[black,->-=.5] (7J3)--(7J2);
\draw[black,->-=.5] (7J4)--(7J5);
\draw[black,->-=.5] (7J5)--(7J6);
\draw[black,->-=.5] (7J6)--(7J4);
\draw[black,->-=.5] (7Jb)--(7J7);
\draw[black,->-=.5] (7Ja1)--(7J2);
\draw[black,->-=.5] (7Ja2)--(7J1);
\draw[black,->-=.5] (7J3)--(7Jb);
\draw[black,->-=.5] (7J7)--(7J3);

\node[minimum size=.1pt,label={left:7J}](7J)at(9.2,-10.5){};

\node[enclosed,label={right,yshift=.2cm:6}](8A6)at(2.4,-11.5){};
\node[enclosed,label={right,yshift=.2cm:7}](8A7)at(3.8,-13.2){};
\node[enclosed,label={right,yshift=.2cm:8}](8A8)at(3.2,-15){};
\node[enclosed,label={right,yshift=.2cm:1}](8A1)at(.7,-16.45){};
\node[enclosed,label={left,yshift=.2cm:3}](8A3)at(-.2,-15){};
\node[enclosed,label={left,yshift=.2cm:4}](8A4)at(-.8,-13.2){};
\node[enclosed,label={left,yshift=.2cm:5}](8A5)at(.6,-11.5){};
\node[enclosed,label={left,yshift=.2cm:2}](8A2)at(2.3,-16.45){};
\node[minimum size=.1pt](8Aa)at(1.5,-15.9){};
\node[ellipse,minimum width=2.5cm,minimum height=1.3cm,draw](8Ab)at(1.5,-16.45){};

\draw[black,->-=.5] (8A1)--(8A2);
\draw[black,->-=.5] (8Aa)--(8A3);
\draw[black,->-=.6] (8Aa)--(8A4);
\draw[black,->-=.5] (8Aa)--(8A6);
\draw[black,->-=.5] (8A3)--(8A4);
\draw[black,->-=.6] (8A3)--(8A5);
\draw[black,->-=.25] (8A3)--(8A7);
\draw[black,->-=.5] (8A4)--(8A5);
\draw[black,->-=.7] (8A4)--(8A6);
\draw[black,->-=.43] (8A4)--(8A8);
\draw[black,->-=.5] (8A5)--(8A6);
\draw[black,->-=.8] (8A5)--(8A7);
\draw[black,->-=.54] (8A5)--(8Aa);
\draw[black,->-=.5] (8A6)--(8A7);
\draw[black,->-=.7] (8A6)--(8A8);
\draw[black,->-=.42] (8A6)--(8A3);
\draw[black,->-=.5] (8A7)--(8A8);
\draw[black,->-=.2] (8A7)--(8Aa);
\draw[black,->-=.5] (8A7)--(8A4);
\draw[black,->-=.5] (8A8)--(8Aa);
\draw[black,->-=.4] (8A8)--(8A3);
\draw[black,->-=.25] (8A8)--(8A5);

\node[minimum size=.1pt,label={left:8A}](8A)at(2.3,-17.5){};

\node[ellipse,minimum width=4cm,minimum height=6cm,draw](8Ba)at(10.5,-13.9){};
\node[ellipse,minimum width=3cm,minimum height=3cm,draw](8Bb)at(10.5,-12.9){};
\node[enclosed,label={left,yshift=.2cm:1}](8B1)at(6.2,-14.5){};

\node[ellipse,minimum width=2.5cm,minimum height=1cm,draw](8Bc)at(6.2,-13.3){};
\node[enclosed,label={right,yshift=.01cm:2}](8B2)at(6.8,-13.3){};
\node[enclosed,label={left,yshift=.01cm:8}](8B8)at(5.6,-13.3){};

\node[enclosed,label={right,yshift=.2cm:3}](8B3)at(7.4,-13.9){};
\node[enclosed,label={left,yshift=.2cm:4}](8B4)at(10,-12.3){};
\node[enclosed,label={right,yshift=.2cm:5}](8B5)at(11,-12.8){};
\node[enclosed,label={left,yshift=-.2cm:6}](8B6)at(10,-13.3){};
\node[enclosed,label={right,yshift=.2cm:7}](8B7)at(10.5,-15.5){};
\node[minimum size=.1pt](8Ba1)at(8.9,-12.5){};
\node[minimum size=.1pt](8Ba2)at(8.9,-15.5){};

\draw[black,->-=.5] (8Bc)--(8B1);
\draw[black,->-=.5] (8B1)--(8B3);
\draw[black,->-=.5] (8B3)--(8Bc);
\draw[black,->-=.5] (8B4)--(8B5);
\draw[black,->-=.5] (8B5)--(8B6);
\draw[black,->-=.5] (8B6)--(8B4);
\draw[black,->-=.5] (8Bb)--(8B7);
\draw[black,->-=.5] (8Ba1)--(8Bc);
\draw[black,->-=.5] (8Ba2)--(8B1);
\draw[black,->-=.5] (8B3)--(8Bb);
\draw[black,->-=.5] (8B7)--(8B3);
\draw[black,->-=.5] (8B8)--(8B2);

\node[minimum size=.1pt,label={left:8B}](8B)at(9.2,-17.5){};

\end{tikzpicture}}
\end{adjustbox}
\caption{Finite exceptions -2-}
\label{figure 2}
\end{figure}
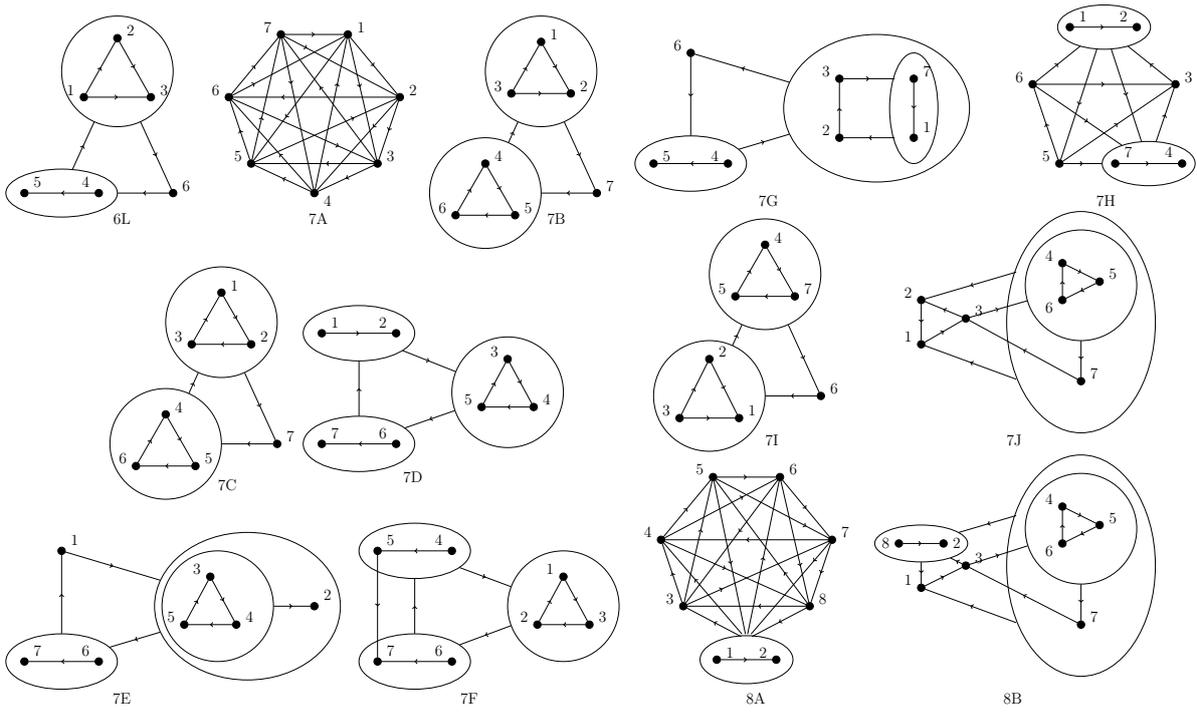

\begin{figure}
\begin{adjustbox}{max width=\textwidth, max height=\textheight}
\scalebox{0.5}{
\begin{tikzpicture}
\tikzset{enclosed/.style={draw,circle,inner sep=2pt,minimum size=4pt,fill=black}}
\tikzset{->-/.style={decoration={
            markings,
            mark=at position #1 with
            {\arrow{>}}},postaction={decorate}}}
            
\node[ellipse,minimum width=2.5cm,minimum height=2.5cm,draw](a)at(0,0){};
\node[enclosed,label={left,yshift=.2cm:1}](1)at(-0.6,-0.4){};
\node[enclosed,label={left,yshift=.2cm:2}](2)at(0,0.55){};
\node[enclosed,label={right,yshift=.2cm:3}](3)at(0.6,-0.4){};
\node[ellipse,minimum width=2.5cm,minimum height=2.5cm,draw](b)at(0,-3.5){};
\node[minimum size=.1pt,label={left:X}](X)at(0.5,-3.5){};
\draw[black,->-=.5] (1)--(2);
\draw[black,->-=.5] (2)--(3);
\draw[black,->-=.5] (3)--(1);
\draw[black,->-=.5] (b)--(a);

\node[minimum size=.1pt,label={left:$F_1(n)$}](F_1)at(1,-5.2){};

\node[ellipse,minimum width=4.2cm,minimum height=2.5cm,draw](2Fa)at(3.8,0){};
\node[ellipse,minimum width=2cm,minimum height=0.9cm,draw](2Fc)at(4.6,-0.5){};
\node[enclosed,label={left,yshift=.2cm:1}](2F1)at(2.8,-0.5){};
\node[enclosed,label={left,yshift=.2cm:2}](2F2)at(3.8,0.6){};
\node[enclosed,label={left,yshift=.01cm:3}](2F3)at(4.2,-0.5){};
\node[enclosed,label={right,yshift=.01cm:4}](2F4)at(5.05,-0.5){};
\node[ellipse,minimum width=2.5cm,minimum height=2.5cm,draw](2Fb)at(3.8,-3.5){};
\node[minimum size=.1pt,label={left:X}](2FX)at(4.3,-3.5){};
\draw[black,->-=.5] (2F1)--(2F2);
\draw[black,->-=.5] (2F2)--(2Fc);
\draw[black,->-=.5] (2Fc)--(2F1);
\draw[black,->-=.5] (2F4)--(2F3);
\draw[black,->-=.5] (2Fb)--(2Fa);

\node[minimum size=.1pt,label={left:$F_2(n)$}](F_2)at(4.6,-5.2){};

\node[ellipse,minimum width=2.5cm,minimum height=2.5cm,draw](3Fa)at(9.3,-1.5){};
\node[minimum size=.1pt,label={left:X}](3FX)at(9.7,-1.5){};
\node[enclosed,label={left,yshift=.2cm:1}](3F1)at(7.85,-3.6){};
\node[enclosed,label={left,yshift=.2cm:2}](3F2)at(7.1,-1.9){};
\node[enclosed,label={right,yshift=.01cm:3}](3F3)at(9.3,-2.3){};
\node[minimum size=.1pt](3Fa1)at(8.3,-1.9){};
\node[minimum size=.1pt](3Fa2)at(8.7,-2.4){};

\draw[black,->-=.5] (3F2)--(3F1);
\draw[black,->-=.5] (3Fa2)--(3F1);
\draw[black,->-=.5] (3Fa1)--(3F2);
\draw[black,->-=.5] (3F1) to [out=0,in=-110,looseness=0.8] (3F3);

\node[minimum size=.1pt,label={left:$F_3(n)$}](F_3)at(9.3,-5.2){};

\node[ellipse,minimum width=2.5cm,minimum height=2.5cm,draw](4Fa)at(13.8,-1.5){};
\node[minimum size=.1pt,label={left:X}](4FX)at(14.2,-1.4){};
\node[enclosed,label={left,yshift=.01cm:1}](4F1)at(11.85,-3.6){};
\node[enclosed,label={left,yshift=.2cm:2}](4F2)at(11.6,-1.9){};
\node[enclosed,label={right,yshift=.01cm:3}](4F3)at(13.8,-2.3){};
\node[minimum size=.1pt](4Fa1)at(12.8,-1.9){};
\node[minimum size=.1pt](4Fa2)at(13.2,-2.4){};
\node[ellipse,minimum width=2.2cm,minimum height=1cm,draw](4Fb)at(12.35,-3.6){};
\node[enclosed,label={right,yshift=.01cm:4}](4F4)at(12.85,-3.6){};

\draw[black,->-=.5] (4F2)--(4Fb);
\draw[black,->-=.5] (4Fa2)--(4Fb);
\draw[black,->-=.5] (4Fa1)--(4F2);
\draw[black,->-=.5] (4Fb) to [out=10,in=-70,looseness=0.9] (4F3);
\draw[black,->-=.5] (4F1)--(4F4);

\node[minimum size=.1pt,label={left:$F_4(n)$}](F_4)at(13,-5.2){};

\node[ellipse,minimum width=2.5cm,minimum height=2.5cm,draw](5Fa)at(3.5,-7.2){};
\node[minimum size=.1pt,label={left:X}](5FX)at(4,-7.2){};
\node[enclosed,label={left,yshift=.01cm:1}](5F1)at(1.75,-9.2){};
\node[ellipse,minimum width=2.5cm,minimum height=2.5cm,draw](5Fb)at(0,-7.2){};
\node[minimum size=.1pt,label={left:Y}](5FY)at(.5,-7.2){};
\node[enclosed,label={right,yshift=.01cm:2}](5F2)at(3.5,-8){};
\node[minimum size=.1pt](5Fa1)at(2.6,-7.9){};
\node[minimum size=.1pt](5Fb1)at(.9,-7.9){};

\draw[black,->-=.5] (5Fa)--(5Fb);
\draw[black,->-=.5] (5Fa1)--(5F1);
\draw[black,->-=.5] (5Fb1)--(5F1);
\draw[black,->-=.5] (5F1) to [out=0,in=-110,looseness=0.8] (5F2);

\node[minimum size=.1pt,label={left:$F_5(n)$}](F_5)at(2.3,-10.5){};

\node[ellipse,minimum width=2.5cm,minimum height=2.5cm,draw](a)at(9.7,-7.2){};
\node[minimum size=.1pt,label={left:X}](X)at(10.2,-7.2){};
\node[enclosed,label={left,yshift=.01cm:3}](3)at(7.4,-9.2){};
\node[enclosed,label={right,yshift=.01cm:2}](2)at(9.7,-8){};
\node[minimum size=.1pt](a1)at(8.8,-7.9){};
\node[minimum size=.1pt](c1)at(7.1,-7.9){};
\node[ellipse,minimum width=2.2cm,minimum height=1cm,draw](b)at(7.95,-9.2){};
\node[enclosed,label={right,yshift=.01cm:1}](1)at(8.5,-9.2){};
\node[ellipse,minimum width=2.5cm,minimum height=2.5cm,draw](c)at(6.2,-7.2){};
\node[minimum size=.1pt,label={left:Y}](Y)at(6.7,-7.2){};

\draw[black,->-=.5] (3)--(1);
\draw[black,->-=.5] (a1)--(b);
\draw[black,->-=.5] (c1)--(b);
\draw[black,->-=.5] (b) to [out=10,in=-90,looseness=0.9] (2);
\draw[black,->-=.5] (3)--(1);
\draw[black,->-=.5] (a)--(c);

\node[minimum size=.1pt,label={left:$F_6(n)$}](F_6)at(8,-10.5){};

\node[ellipse,minimum width=2.5cm,minimum height=2.5cm,draw](a)at(13.8,-7.2){};
\node[minimum size=.1pt,label={left:Y}](Y)at(14.3,-7.2){};
\node[enclosed,label={left,yshift=.01cm:3}](3)at(11.85,-9.2){};
\node[enclosed,label={left,yshift=.2cm:1}](1)at(11.6,-7.7){};
\node[minimum size=.1pt](a1)at(12.8,-7.7){};
\node[minimum size=.1pt](a2)at(13.1,-8){};
\node[ellipse,minimum width=2.2cm,minimum height=1cm,draw](b)at(12.35,-9.2){};
\node[enclosed,label={right,yshift=.01cm:2}](2)at(12.85,-9.2){};

\draw[black,->-=.5] (1)--(a1);
\draw[black,->-=.5] (a2)--(b);
\draw[black,->-=.5] (b)--(1);
\draw[black,->-=.5] (3)--(2);

\node[minimum size=.1pt,label={left:$F_7(n)$}](F_7)at(13,-10.5){};

\node[ellipse,minimum width=2.5cm,minimum height=2.5cm,draw](a)at(0,-12.5){};
\node[ellipse,minimum width=2.5cm,minimum height=2.5cm,draw](b)at(0,-16){};
\node[minimum size=.1pt,label={left:3A}](3A)at(0.5,-16){};
\node[minimum size=.1pt,label={left:X}](X)at(0.45,-12.5){};

\draw[black,->-=.5] (b)--(a);

\node[minimum size=.1pt,label={left:$F_8(n)$}](F_8)at(0.65,-17.8){};

\node[ellipse,minimum width=2.5cm,minimum height=2.5cm,draw](a)at(3.2,-12.5){};
\node[ellipse,minimum width=2.5cm,minimum height=2.5cm,draw](b)at(3.2,-16){};
\node[minimum size=.1pt,label={left:5A}](5A)at(3.7,-16){};
\node[minimum size=.1pt,label={left:X}](X)at(3.65,-12.5){};

\draw[black,->-=.5] (b)--(a);

\node[minimum size=.1pt,label={left:$F_9(n)$}](F_9)at(3.85,-17.8){};

\node[ellipse,minimum width=2.5cm,minimum height=2.5cm,draw](a)at(6.4,-12.5){};
\node[ellipse,minimum width=2.5cm,minimum height=2.5cm,draw](b)at(6.4,-16){};
\node[minimum size=.1pt,label={left:7A}](7A)at(6.9,-16){};
\node[minimum size=.1pt,label={left:X}](X)at(6.85,-12.5){};

\draw[black,->-=.5] (b)--(a);

\node[minimum size=.1pt,label={left:$F_{10}(n)$}](F_10)at(7.05,-17.8){};

\node[ellipse,minimum width=4cm,minimum height=5.5cm,draw](a)at(10.6,-14.5){};
\node[enclosed,label={left,yshift=.2cm:1}](1)at(10.6,-13){};

\node[ellipse,minimum width=3cm,minimum height=3cm,draw](b)at(10.6,-15.5){};
\node[minimum size=.1pt,label={left:X}](X)at(11.1,-15.5){};

\node[minimum size=.1pt](a1)at(12.4,-13.7){};
\node[minimum size=.1pt](a2)at(12.4,-15.3){};

\node[enclosed,label={right,yshift=.2cm:2}](2)at(13.6,-13.7){};
\node[enclosed,label={right,yshift=.2cm:3}](3)at(13.6,-15.3){};

\draw[black,->-=.5] (b)--(1);
\draw[black,->-=.5] (2)--(3);
\draw[black,->-=.5] (3)--(a2);
\draw[black,->-=.5] (a1)--(2);

\node[minimum size=.1pt,label={left:$F_{11}(n)$}](F_11)at(13,-17.8){};

\end{tikzpicture}}
\scalebox{0.5}{
\begin{tikzpicture}
\tikzset{enclosed/.style={draw,circle,inner sep=2pt,minimum size=4pt,fill=black}}
\tikzset{->-/.style={decoration={
            markings,
            mark=at position #1 with
            {\arrow{>}}},postaction={decorate}}}
            
\node[ellipse,minimum width=4cm,minimum height=5.5cm,draw](a)at(0,0){};
\node[ellipse,minimum width=2.2cm,minimum height=1cm,draw](c)at(0,1.5){};
\node[enclosed,label={right,yshift=.01cm:1}](1)at(0.5,1.5){};
\node[enclosed,label={left,yshift=.01cm:4}](4)at(-0.5,1.5){};

\node[ellipse,minimum width=3cm,minimum height=3cm,draw](b)at(0,-1){};
\node[minimum size=.1pt,label={left:X}](X)at(0.5,-1){};

\node[minimum size=.1pt](a1)at(1.8,0.8){};
\node[minimum size=.1pt](a2)at(1.8,-0.8){};

\node[enclosed,label={right,yshift=.2cm:2}](2)at(3,0.8){};
\node[enclosed,label={right,yshift=.2cm:3}](3)at(3,-0.8){};

\draw[black,->-=.5] (b)--(c);
\draw[black,->-=.5] (4)--(1);
\draw[black,->-=.5] (2)--(3);
\draw[black,->-=.5] (3)--(a2);
\draw[black,->-=.5] (a1)--(2);

\node[minimum size=.1pt,label={left:$F_{12}(n)$}](F_12)at(0.65,-3.5){};

\node[ellipse,minimum width=4cm,minimum height=5.5cm,draw](a)at(8,0){};
\node[enclosed,label={left,yshift=.2cm:1}](1)at(8,1.3){};

\node[ellipse,minimum width=3cm,minimum height=3cm,draw](b)at(8,-1){};
\node[minimum size=.1pt,label={left:X}](X)at(8.5,-1){};

\node[minimum size=.1pt](a1)at(9.7,1.2){};
\node[minimum size=.1pt](a2)at(9.8,-0.8){};

\node[enclosed,label={right,yshift=.2cm:2}](2)at(12,1.2){};

\node[ellipse,minimum width=3cm,minimum height=3cm,draw](c)at(12,-0.8){};
\node[enclosed,label={right,yshift=.2cm:3}](3)at(12.7,-1.3){};
\node[enclosed,label={left,yshift=.2cm:4}](4)at(11.3,-1.3){};
\node[enclosed,label={right,yshift=.2cm:5}](5)at(12,0){};

\draw[black,->-=.5] (b)--(1);
\draw[black,->-=.5] (2)--(c);
\draw[black,->-=.5] (c)--(a2);
\draw[black,->-=.5] (a1)--(2);
\draw[black,->-=.5] (3)--(4);
\draw[black,->-=.5] (4)--(5);
\draw[black,->-=.5] (5)--(3);

\node[minimum size=.1pt,label={left:$F_{13}(n)$}](F_13)at(8.65,-3.5){};

\node[ellipse,minimum width=4cm,minimum height=5.5cm,draw](a)at(5,-7){};
\node[ellipse,minimum width=2.2cm,minimum height=1cm,draw](c)at(5,-5.5){};
\node[enclosed,label={right,yshift=.01cm:1}](1)at(5.5,-5.5){};
\node[enclosed,label={left,yshift=.01cm:6}](6)at(4.5,-5.5){};

\node[ellipse,minimum width=3cm,minimum height=3cm,draw](b)at(5,-8){};
\node[minimum size=.1pt,label={left:X}](X)at(5.5,-8){};

\node[minimum size=.1pt](a1)at(6.7,-5.8){};
\node[minimum size=.1pt](a2)at(6.8,-7.8){};

\node[enclosed,label={right,yshift=.2cm:2}](2)at(9,-5.8){};

\node[ellipse,minimum width=3cm,minimum height=3cm,draw](d)at(9,-7.8){};
\node[enclosed,label={right,yshift=.2cm:3}](3)at(9.7,-8.3){};
\node[enclosed,label={left,yshift=.2cm:4}](4)at(8.3,-8.3){};
\node[enclosed,label={right,yshift=.2cm:5}](5)at(9,-7){};

\draw[black,->-=.5] (b)--(c);
\draw[black,->-=.5] (6)--(1);
\draw[black,->-=.5] (2)--(d);
\draw[black,->-=.5] (d)--(a2);
\draw[black,->-=.5] (a1)--(2);
\draw[black,->-=.5] (3)--(4);
\draw[black,->-=.5] (4)--(5);
\draw[black,->-=.5] (5)--(3);

\node[minimum size=.1pt,label={left:$F_{14}(n)$}](7I)at(5.65,-10.5){};

\end{tikzpicture}}
\end{adjustbox}
\caption{The infinite families of exceptions}
\label{figure 3}

\end{figure}
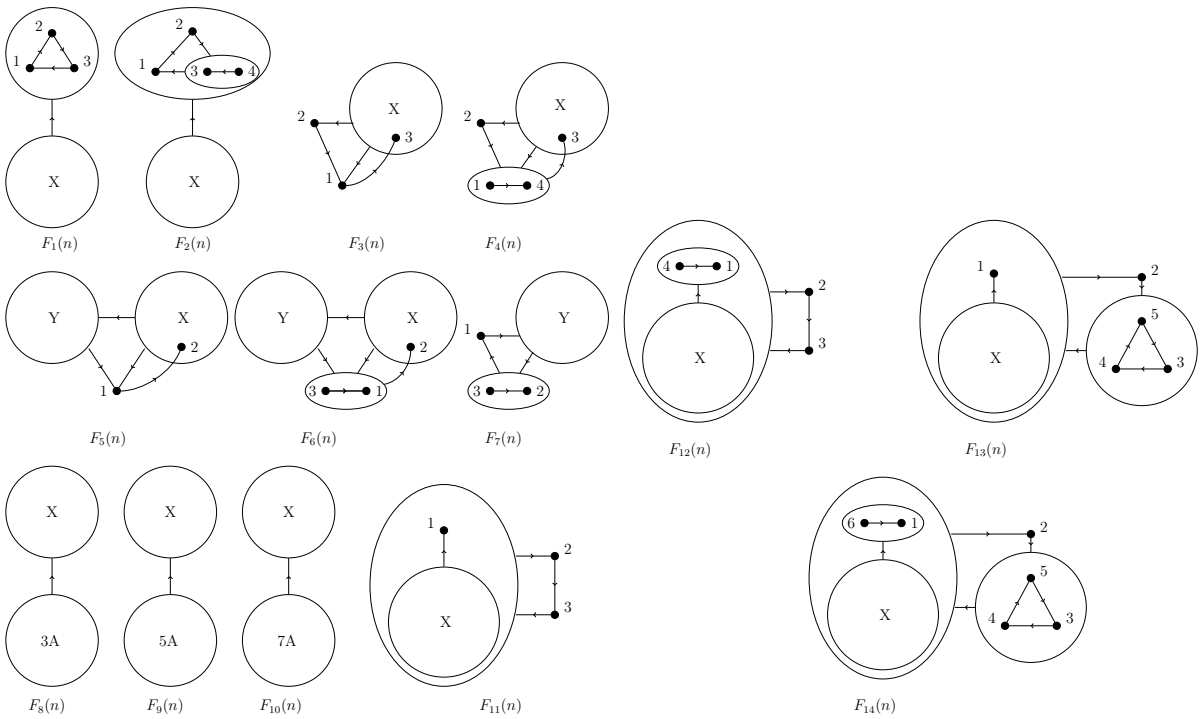

The infinite families of exceptions are denoted by $E_i(n)=(F_i(n);P)$, where $F_i(n)$ is the tournament on $n$ vertices illustrated in figure \ref{figure 3}. For each $E_i(n)$, we define the set $S$ of vertices of $F_i(n)$ which are not origin of $P$ together with the conditions on the tournament. And finally, we give the paths $P$ with origin $x\notin S$. Below the list of infinite families of exceptions:\\
Exception $E_1(n)=(F_1(n),(1,n-2))$; $S=\{1,2,3\}$. Conditions: $|X|\geq 1$. Paths: for any $u\in X,\,P=u132I_{X-u}$.
\\Exception $E_2(n)=(F_2(n),(2,n-3))$; $S=\{3,4\}$. Conditions: $|X|\geq 1$. Paths: $P=1234I_X$, $P=2314I_X$ and for any $u\in X,\,P=u132I_{X-u}$.
\\Exception $E_3(n)=(F_3(n),(1,n-2))$; $S=\{1,3\}$. Conditions: $N^+(3)\neq\{2\}$ and 3 is an ingenerator of $T(X)$ (in particular there exists a Hamiltonian directed inpath $u3R_1$ of $T(X)$ and another $3vR_2$ of $T(X)$). Paths: $P=21u3R_1$ and for any $y\in X\setminus \{3\}$, $P=y132I_{X-y}$.
\\Exception $E_4(n)=(F_4(n),(2,n-3))$; $S=\{1,4\}$. Conditions: $N^+(3)\neq\{2\}$ and 3 is an ingenerator of $T(X)$ (in particular there exists a Hamiltonian directed inpath $u3R_1$ of $T(X)$ and another $3vR_2$ of $T(X)$). Paths: $P=214u3R_1$, $P=3241vR_2$ and for any $y\in X\setminus \{3\}$, $P=u132I_{X-u}$.
\\Exception $E_5(n)=(F_5(n),(1,n-2))$; $S=\{1,2\}$. Conditions: $n\geq 5$, $|Y|\geq 2$ and 2 is an ingenerator of $T(X)$. Paths: for any $x\in X\setminus \{2\}$, $P=x1I_YI_{X-x}$, and for any $y\in Y$, $P=y1I_{Y-y}I_X$.
\\Exception $E_6(n)=(F_6(n),(2,n-3))$; $S=\{1,3\}$. Conditions: $|Y|\geq 2$ and 2 is an ingenerator of $T(X)$. Paths: for any $x\in X$, $P=xz13I_{Y-z}I_{X-x}$ (for a given $z\in Y$) and for any $y\in Y$, $P=y31I_{Y-y}I_X$.
\\Exception $E_7(n)=(F_7(n),(1,1,n-3))$; $S=\{2,3\}$. Conditions: $T(Y)$ is not $3-$cycle and $|Y|\geq 3$. Paths: since $T(Y)$ is not a $3-$cycle, there is a path $Q=-(1,n-5)$ in $T(Y)$ (this is clear if $T(Y)$ is not reducible; if $T(Y)$ is strong, there exists a vertex $w\in Y$ such that $T(Y)-w$ is also strong, such a vertex is certainly an origin of $Q$ since it has an inneighbor that is an origin of a Hamiltonian directed path of $T(Y)-w$), we then have $P=1Q23$ and for any $y\in Y$, $P=y231O_{Y-y}$.
\\Exception $E_8(n)=(F_8(n),(n-4,1,1,1))$; $S=X$. Conditions: $3A$ is the $3-$cycle, its set of vertices is $\{1,2,3\}$, we furthermore need $|X|\geq 2$. Paths: $P=1O_{X-u}2u3$, $P=2O_{X-u}1u3$ and $P=3O_{X-u}2u1$ for a given $u\in X$.
\\Exception $E'_8(n)=(F_8(n),(n-4,2,1))$; $S=X$. Conditions: $3A$ is the $3-$cycle, its set of vertices is $\{1,2,3\}$, we furthermore need $|X|\geq 2$. Paths: $P=1O_{X-u}32u$, $P=2O_{X-u}13u$ and $P=3O_{X-u}21u$ for a given $u\in X$.
\\Exception $E_9(n)=(F_9(n),(n-6,1,1,1,1,1))$; $S=X$. Conditions: $5A$ is the $2-$regular tournament, its set of vertices is $\{1,2,3,4,5\}$, we furthermore need $|X|\geq 2$. Paths: $P=1O_{X-u}2u453$, $P=2O_{X-u}3u514$, $P=3O_{X-u}4u125$, $P=4O_{X-u}5u231$ and $P=5O_{X-u}1u342$ for a given $u\in X$.
\\Exception $E'_9(n)=(F_9(n),(n-6,2,1,1,1))$; $S=X$. Conditions: $5A$ is the $2-$regular tournament, its set of vertices is $\{1,2,3,4,5\}$, we furthermore need $|X|\geq 2$. Paths: $P=1O_{X-u}32u45$, $P=2O_{X-u}43u51$, $P=3O_{X-u}54u12$, $P=4O_{X-u}15u23$ and $P=5O_{X-u}21u34$ for a given $u\in X$.
\\Exception $E_{10}(n)=(F_{10}(n),(n-8,1,1,1,1,1,1,1))$; $S=X$. Conditions: $7A$ is the Paley tournament, its set of vertices is $\{1,2,3,4,5,6,7\}$, we furthermore need $|X|\geq 2$. Paths: $P=1O_{X-u}2u45376$, $P=2O_{X-u}3u56417$, $P=3O_{X-u}4u67521$, $P=4O_{X-u}5u71632$, $P=5O_{X-u}6u12743$, $P=6O_{X-u}7u23154$ and $P=7O_{X-u}1u34265$ for a given $u\in X$.
\\Exception $E'_{10}(n)=(F_{10}(n),(n-8,2,1,1,1,1,1))$; $S=X$. Conditions: $7A$ is the Paley tournament, its set of vertices is $\{1,2,3,4,5,6,7\}$, we furthermore need $|X|\geq 2$. Paths: $P=1O_{X-u}32u4657$, $P=2O_{X-u}43u5761$, $P=3O_{X-u}54u6172$, $P=4O_{X-u}65u7213$, $P=5O_{X-u}76u1324$, $P=6O_{X-u}17u2435$ and $P=7O_{X-u}21u3546$ for a given $u\in X$.
\\Exception $E_{11}(n)=(F_{11}(n),(1,1,n-3))$; $S=\{1,2\}$. Conditions: $|X|\geq 2$. Paths: $P=31O_X2$ and for any $u\in X,\,P=u1O_{X-u}23$.
\\Exception $E_{12}(n)=(F_12(n),(2,1,n-4))$; $S=\{1,4\}$. Conditions: $|X|\geq 2$. Paths: $P=231O_X4$ and for any $u\in X,\,P=u41O_{X-u}23$.
\\Exception $E_{13}(n)=(F_{13}(n),(1,1,n-3))$; $S=\{1,2\}$. Conditions: $|X|\geq 2$. Paths: $P=3425O_X1$, $P=4523O_X1$, $P=5324O_X1$ and for all $u,v\in X,\,P=u1v2345O_{X\setminus \{u,v\}}$.
\\Exception $E_{14}(n)=(F_{14}(n),(2,1,n-4))$; $S=\{1,6\}$. Conditions: $|X|\geq 2$. Paths: for every vertex $u\in X$, $P=24u53O_{X-u}61$, $P=3u16245O_{X-u}$, $P=4u16253O_{X-u}$, $P=5u16234O_{X-u}$ and $P=u61345O_{X-u}2$.

\bibliographystyle{abbrv}
\bibliography{references}

\end{document}